\documentclass{article}

\usepackage{amsmath, amssymb, amsthm,amsfonts}
\usepackage{centernot}
\usepackage{mathpazo}
\usepackage[margin=2.5cm]{geometry}

\usepackage{color, graphicx, soul}
\usepackage{subfig}
\usepackage{adjustbox}
\usepackage{dsfont} 
\usepackage{resizegather}

\usepackage{xargs}
\usepackage{cite,microtype}
\usepackage{dsfont} 
\usepackage{xr}

\usepackage[colorinlistoftodos,prependcaption,textsize=footnotesize]{todonotes}
\newcommandx{\attn}[2][1=]{\todo[linecolor=red,backgroundcolor=blue!25,bordercolor=red,#1]{#2}}
\newcommandx{\other}[2][1=]{\todo[linecolor=OliveGreen,backgroundcolor=OliveGreen!25,bordercolor=OliveGreen,#1]{#2}}
\newcommandx{\thiswillnotshow}[2][1=]{\todo[disable,#1]{#2}}

\usepackage{hyperref}
\hypersetup{
	colorlinks,
	linkcolor={red!50!black},
	citecolor={blue!50!black},
	urlcolor={blue!80!black}
}

\ifpdf
\hypersetup{
	pdftitle={Power Cone Error Bounds},
	pdfauthor={Y. Lin, S.B. Lindstrom, B.F. Louren\c{c}o, and T.K. Pong}
}
\fi

\usepackage{enumitem}
\usepackage{amssymb}
\usepackage{color,bm}
\usepackage{dsfont}

\newcommand*\xbar[1]{%
	\hbox{%
		\vbox{%
			\hrule height 0.5pt 
			\kern0.5ex
			\hbox{%
				\kern-0.1em
				\ensuremath{#1}%
				\kern-0.1em
			}%
		}%
	}%
}

\def\argmin{\mathop{\rm arg\,min}}

\def\dist{{\rm dist}}
\def\ri{{\rm ri\,}}
\def\Span{{\rm span\,}}
\def\dim{{\rm dim\,}}

\def\lie{{\rm Lie\,}}
\def\aut{{\rm Aut\,}}

\def\K{{{\cal P}^{^{\bm{\alpha}}}_{1,n}}}
\def\dK{({{\cal P}^{^{\bm{\alpha}}}_{1,n}})^*}
\def\gK{{{\cal P}^{^{\bm{\alpha}}}_{m,n}}}
\def\dgK{({{\cal P}^{^{\bm{\alpha}}}_{m,n}})^*}
\def\bd{\partial}
\def\tx{\widetilde{x}}
\def\ty{\widetilde{y}}
\def\tz{\widetilde{z}}
\def\tv{\widetilde{v}}
\def\tu{\widetilde{u}}
\def\tw{\widetilde{w}}
\def\tbx{\widetilde{\bm{x}}}
\def\tby{\widetilde{\bm{y}}}
\def\tbz{\widetilde{\bm{z}}}
\def\tbv{\widetilde{\bm{v}}}

\def\bz{\overline{z}}

\def\bv{\overline{v}}

\def\bbx{\overline{\bm{x}}}
\def\bby{\overline{\bm{y}}}
\def\bbz{\overline{\bm{z}}}

\def\bbv{\overline{\bm{v}}}
\def\bbw{\overline{\bm{w}}}

\def\R{{\rm I\!R}}
\def\C{\mathcal{C}}
\def\B{\mathcal{B}}
\def\Fhatz{{{\cal F}_{\bm{z}}}}
\def\Fr{{{\cal F}_{\rm r}}}
\def\dpps{d_{\rm PPS}}
\def\stdCone{{\cal K}}
\def\stdFace{{\cal F}}
\def\stdIndex{{\cal I}}
\def\subSpace{{\cal L}}
\def\EucSpace{{\cal E}}
\def\diag{{\rm Diag}}

\newcommand{\inProd}[2]{\langle #1 , #2 \rangle }
\newcommand{\gKmn}[2]{{{\cal P}^{^{\bm{\alpha}}}_{#1,#2}} }

\makeatletter
\DeclareRobustCommand{\properideal}{\mathrel{\text{\( \m@th\proper@ideal \)}}}
\newcommand{\proper@ideal}{%
	\ooalign{\( \lneq \)\cr\raise.22ex\hbox{\( \lhd \)}\cr}%
}
\makeatother

\newtheorem{theorem}{Theorem}[section]
\newtheorem{lemma}[theorem]{Lemma}
\newtheorem{corollary}[theorem]{Corollary}
\newtheorem{proposition}[theorem]{Proposition}

\theoremstyle{definition}
\newtheorem{definition}[theorem]{Definition}
\theoremstyle{definition}

\theoremstyle{definition}
\newtheorem{remark}[theorem]{Remark}
\numberwithin{equation}{section}

\begin{document}
	\title{Generalized power cones: optimal error bounds and automorphisms}
\author{
	Ying Lin\thanks{Department of Applied Mathematics, the Hong Kong Polytechnic University, Hong Kong, People's Republic of China. E-mail: \href{ying.lin@connect.polyu.hk}{ying.lin@connect.polyu.hk}.}
	\and 
	Scott B.\ Lindstrom\thanks{
		Centre for Optimisation and Decision Science, Curtin University, Australia.
		E-mail: \href{scott.lindstrom@curtin.edu.au}{scott.lindstrom@curtin.edu.au}.}
	\and
	Bruno F. Louren\c{c}o\thanks{Department of Statistical Inference and Mathematics, Institute of Statistical Mathematics, Japan.
		This author was supported partly by the JSPS  Grant-in-Aid for Early-Career Scientists  19K20217, 23K16844 and the Grant-in-Aid for Scientific Research (B)21H03398.
		Email: \href{bruno@ism.ac.jp}{bruno@ism.ac.jp}.}
	\and
	Ting Kei Pong\thanks{
		Department of Applied Mathematics, the Hong Kong Polytechnic University, Hong Kong, People's Republic of China.
		This author was supported partly by the Hong Kong Research Grants Council PolyU153002/21p.
		E-mail: \href{tk.pong@polyu.edu.hk}{tk.pong@polyu.edu.hk}.
	}
}
	\maketitle

	\begin{abstract}
	  Error bounds are a requisite for trusting or distrusting solutions in an informed way.
	  Until recently, provable error bounds in the absence of constraint qualifications were unattainable for many classes of cones that do not admit projections with known succinct expressions.
	  We build such error bounds for the generalized power cones, using the recently developed framework of one-step facial residual functions.
	  We also show that our error bounds are tight in the sense of that framework.
	  Besides their utility for understanding solution reliability, the error bounds we discover have additional applications to the algebraic structure of the underlying cone, which we describe.
	  In particular we use the error bounds to compute the dimension of the automorphism group for the generalized power cones, and to identify a set of generalized power cones that are self-dual, irreducible, nonhomogeneous, and perfect.
	\end{abstract}

{\small
	
}

	\section{Introduction}\label{sec:intro}
	In a Euclidean space \( \EucSpace \), consider the conic feasibility problem:
	\begin{equation}
		\label{eq:conic-feasibility-problem}\tag{Feas}
		\text{find } \quad \bm{x} \in (\subSpace + \bm{a}) \cap \stdCone,
	\end{equation}
	where \( \subSpace \) is a subspace, \( \bm{a} \in \EucSpace \), and \( \stdCone\) is a closed convex cone.
	We desire an upper bound on the distance from {an arbitrary} $\bm{x}$ to the feasible region $(\subSpace+ \bm{a})\cap \stdCone$.
	The upper bound we seek should depend on the two distances between $\bm{x}$ and {$\stdCone$, and} between  $\bm{x}$ and $\subSpace+ \bm{a}$ respectively.
	Such a guarantee is a kind of \textit{error bound}; error bounds are a fundamental topic in the optimization literature \cite{HF52,LP98,LT93,Pang97,ZS17} and widely used in convergence analysis of algorithms.
	Typically, \eqref{eq:conic-feasibility-problem} is almost  never solved exactly, instead algorithms and solvers often return an approximate solution.
	Then, error bounds can be used to evaluate the trustworthiness of approximate solutions because they tell us how close they are to the true set of feasible solutions.

	In this paper, we consider the case when $\stdCone = \gK$ is the generalized power cone
	$$
	\gK = \left\{\bm{x} = (\bbx,\tbx)\in \R^{m+n}\,\bigg| \,\|\bbx\|\le \prod_{i=1}^n\tx_i^{\alpha_i},\,\bbx\in \R^m,\,\tbx\in \R^n_+\right\},
	$$
	where $m\geq 1$, $n \geq 2$, \( \bm{\alpha} = (\alpha_1,\ldots,\alpha_n)\in (0,1)^n \) with \( \sum_{i=1}^n\alpha_i = 1 \), and $\|\bbx\|$ denotes the Euclidean norm of $\bbx$.
	In the specific case when $m\geq 1$, $n=2$, and $\bm{\alpha}=(1 / 2, 1 / 2)$, $\gK$ is isomorphic to a second-order cone, whose worst-case error bound is known to be H\"{o}lderian with exponent $1/2$, thanks to the work of Luo and Sturm \cite{LS00}.
	The remaining cases, while not as well-known as the second-order cone case, admit more direct modeling of certain problems 
	and have found applications in geometric programs, generalized location problems, and portfolio optimization \cite{Ch09,MC2020}.
	More broadly, the inclusion of the power cone\footnote{This refers to $\gKmn{1}{2}$.} (and the exponential cone) makes all the convex instances from the MINLPLib2 benchmark library conic representable \cite{lubin2016extended,MINLPLIB}.
	This broad utility has motivated the development of {self-concordant barriers \cite{Ch09,tunccel2010self,roy2022self}}, and the ongoing development of specialized interior point methods \cite{nesterov2012towards,SY15}.
	Optimization with the generalized power cones is implemented in commercial and open source solvers like MOSEK, Alfonso, DDS and Hypatia \cite{CKV21,KT19,PY21,MC2020}.

	One of this paper's main contributions---Theorem~\ref{thm:error-bound-power-cone-optimality}---is a complete error bound analysis for the generalized power cone problem \eqref{eq:conic-feasibility-problem}. 
	The generalized power cone cases pose {two} significant obstructions to error bound analysis that are not present in the second-order cone case.
	Firstly,  known forms for projections onto generalized power cones do not admit simple representations \cite{hien2015differential}; secondly, their facial structure is more complicated.
	{The first obstruction we obviate via the framework of one-step facial residual functions (\( \mathds{1} \)-FRFs), which was established in \cite{LiLoPo20,LiLoPo21}.}
	The second challenge, facial complexity, we tackle directly.
	In particular, we build \( \mathds{1} \)-FRFs for all faces of $\gK$.
	{All these \( \mathds{1} \)-FRFs are tight in the natural sense of \cite{LiLoPo21}.}
	Consequently, all of the error bounds in Theorem~\ref{thm:error-bound-power-cone-optimality} are tight in this sense.

	While error bounds are typically used in convergence analysis and to evaluate the quality of approximate solutions, our approach via \( \mathds{1} \)-FRFs admits a surprising additional application to the algebraic structure of the underlying cone.
	In order to explain our next results, we recall a few concepts. The \emph{automorphism group} of a cone $\stdCone$ is the set of the bijective linear operators $\bm{A}$ satisfying $\bm{A}\stdCone = \stdCone$.
	A cone is said to be \emph{homogeneous} if its automorphism group acts transitively on its relative interior. We say that a cone is \emph{irreducible} if it is not the direct sum of two nontrivial cones whose spans only intersect at the origin.
	
	Because automorphisms of cones must preserve optimal FRFs (up to positively rescaled shifts), we can use our results to  establish the automorphism group for \( \gK \) in Theorem~\ref{thm:self-duality-homogeneity} and compute its dimension in Theorem~\ref{thm:dimension-Aut(P)}.
		
	This is useful because the automorphism group of a closed convex cone $\stdCone$ has important implications for complementarity problems over $\stdCone$; see \cite{GT2014}.
	In particular, denoting the dual cone of $\stdCone$ by $\stdCone^*$, a complementarity condition of the form ``$\bm{x} \in \stdCone, \bm{y} \in \stdCone^*, \inProd{\bm{x}}{\bm{y}} = 0$''  can be split into a square system of equations if and only if the dimension of the automorphism group of $\stdCone$ is at least $\dim \stdCone$, see \cite[Theorem~1]{OG2016}.
	In this case, $\stdCone$ is said to be a \emph{perfect cone}.
	
	Many of the concrete examples of irreducible perfect cones in the literature correspond to homogeneous cones. In this paper we will show that the generalized power cone is irreducible, perfect (when $m\ge 3$) and, except when it reduces to the second order-cone case, always non-homogeneous. This gives an interesting example of an irreducible cone with good complementarity properties that is not a homogeneous cone.
	To summarize, our main contributions are as follows.
	\begin{enumerate}
		\item We completely determine the tightest possible error bounds for the generalized power cone, see Theorem~\ref{thm:error-bound-power-cone-optimality}.
		\item Using our error bounds, we completely determine the automorphism group of $\gK$ and discuss some theoretical questions related to homogeneity and perfectness (in the sense of \cite{GT14,GT2014}), see Section~\ref{sec:applications}.
	\end{enumerate}
	 Although we do not discuss the details, we mention in passing that determining the error bound associated to conic linear systems makes it possible to compute the KL-exponent of certain functions, as done, for example, in \cite[Section~5.1]{LiLoPo21} using results from \cite{YLP21}. See more on the connection between error bounds, KL exponents and convergence rates in \cite{BNPS17}.

	This paper is organized as follows.
	In Section~\ref{sec:preliminaries}, we recall {notation and preliminaries}.
	In Section~\ref{sec:error-bound-power-cone}, we furnish the eponymous error bounds.
	In Section~\ref{sec:applications}, we provide the further application to the algebraic structure of the generalized power cones.

	\section{Notation and preliminaries}\label{sec:preliminaries}
	We will use plain letters to represent real scalars, bold lowercase letters to denote vectors, bold uppercase letters to stand for matrices,\footnote{With an abuse of notation, we use $\bm 0$ to denote a zero vector / matrix, whose dimension should be clear from the context.} and curly capital letters for (sub)spaces and sets. Let \( \EucSpace \) be a finite dimensional Euclidean space, \( \R_+ \) and \( \R_- \) be the set of nonnegative and nonpositive real numbers, respectively.
	The inner product of \( \EucSpace \) is denoted by \( \inProd{ \cdot }{ \cdot  } \) and the induced norm by \( \| \cdot \| \).
	With that, for \( \bm{x} \in \EucSpace \) and a closed convex set \( \C \subseteq \EucSpace \), we denote the projection of \( \bm{x} \) onto \( \C \) by \( P_\C(\bm{x}) \) so that \( P_\C(\bm{x}) = \argmin_{\bm{y} \in \C} \| \bm{x} - \bm{y} \| \) and the distance between \( \bm{x} \) and \( \C \) by \( \dist(\bm{x}, \C) = \inf_{\bm{y} \in \C} \| \bm{x} - \bm{y} \| = \| \bm{x} - P_\C(\bm{x}) \| \).
	For any \( \bm{x} \in \EucSpace \) and \( \eta \geq 0 \), we denote the ball centered at \( \bm{x} \) with radius \( \eta  \) by \( \B(\bm{x}; \eta) := \{\bm{y} \in \EucSpace \,|\,\| \bm{y} - \bm{x} \| \leq \eta \} \); we write \( \B(\eta) \) for the ball centered at \( \bm{0} \) with radius \( \eta \) for simplicity.
	A diagonal matrix with diagonal vector being \( \bm{x} \) is denoted by \( \diag(\bm{x}) \).
	Meanwhile, we use \( \C^{\perp } \) to denote the orthogonal complement of \( \C \) and \( \bm{I}_n \) to represent the \( n \times n \) identity matrix.

	{
	  We now recall the definition of {Lipschitzian and H\"olderian error bounds}.
	  Let \( \C_1, \C_2 \subseteq \EucSpace \) be closed convex sets with \( \C_1 \cap \C_2 \neq \emptyset  \).
	  We say that \( \C_1, \C_2 \) satisfy a \emph{uniform H\"olderian error bound} with exponent \( \gamma  \in (0, 1] \) if for every bounded set \( \B \subseteq \EucSpace \), there exists a constant \( \kappa _\B \) such that
	  $\dist(\bm{x}, \C_1 \cap \C_2) \!\leq\! \kappa _\B \max\left\{\dist(\bm{x}, \C_1), \dist(\bm{x}, \C_2)\right\}^{\gamma }$ for all $\bm{x} \in \B$.
	  If \( \gamma  = 1 \), then the error bound is said to be \emph{Lipschitzian}.
	}

	Let \( \stdCone \subseteq \EucSpace \) be a closed convex cone and \( \stdCone^{*} \) be its dual cone.
	We will use ${\rm int}\,\stdCone $, \( \ri \stdCone \), \( \bd \stdCone \), \( \Span \stdCone \), \( \dim \stdCone \) to denote the interior, relative interior, boundary, linear span and dimension of \( \stdCone \), respectively.
	If \( \stdCone \cap -\stdCone = \{\bm{0}\} \), we say that \( \stdCone \) is \textit{pointed}.	

	A \textit{face} of \( \stdCone \) is a closed convex cone \( \stdFace \) satisfying \( \stdFace \subseteq \stdCone \) and the property that if \( \bm{x}, \bm{y} \in \stdCone\) and \( \bm{x} + \bm{y} \in \stdFace \), then \( \bm{x}, \bm{y} \in \stdFace \).\footnote{By convention, we discard the empty face.}
	We write \( \stdFace \unlhd \stdCone \) if \( \stdFace \) is a face of \( \stdCone \) and \( \stdFace \properideal \stdCone \) if \( \stdFace \) is a \textit{proper} face of \( \stdCone \), i.e., \( \stdFace \neq \stdCone \).
	A face \( \stdFace \) is said to be \textit{nontrivial} if \( \stdFace \) is proper and \( \stdFace \neq \stdCone \cap -\stdCone \).
	If \( \stdFace = \stdCone \cap \{\bm{z}\}^{\perp } \) for some \( \bm{z} \in \stdCone^{*} \), \( \stdFace \) is called an \textit{exposed face} of \( \stdCone \).
	
	{
	  The facial structure of the closed convex cone $\stdCone$ is important for deducing \textit{error bounds} for \eqref{eq:conic-feasibility-problem}; see the seminal work of Sturm \cite{St00}.
	  Recently, a new framework based on the \textit{facial reduction algorithm} \cite{BW81, Pa13, WM13} and \textit{one-step facial residual functions ($\mathds{1}$-FRFs)} \cite[Definition 3.4]{LiLoPo20} was proposed for establishing error bounds for \eqref{eq:conic-feasibility-problem} without requiring any constraint qualifications; see \cite{L17, LiLoPo20, LiLoPo21}.
	  Next, we present a very brief overview of the framework, for more detailed explanations and the underlying intuition behind the techniques see \cite{LiLoPo20,LiLoPo21}.
	  First, we recall the definition of
	  one-step facial residual functions.
	
	  \begin{definition}[One-step facial residual function {\cite[Definition 3.4]{LiLoPo20}}]\label{def:one-step-facial-residual-function}
	  	Let \( \mathcal{K} \) be a closed convex cone and \( \bm{z} \in \mathcal{K}^{*} \).
	  	Suppose that \( \psi _{\mathcal{K}, \bm{z}}: \R_+ \times \R_+ \to \R_+ \) satisfies:
	  	\begin{enumerate}[label={(\roman*)}]
	  		\item \( \psi _{\mathcal{K}, \bm{z}} \) is nonnegative,  nondecreasing in each argument and for every \( t \in \R_+ \), \( \psi _{\mathcal{K}, \bm{z}} (0, t) = 0 \).
	  		\item The following implication holds for any \( \bm{x} \in \Span \mathcal{K} \) and \( \epsilon \geq 0 \):\vspace{-0.15 cm}
	  		\[ \vspace{-0.15 cm}
	  		\dist(\bm{x}, \mathcal{K}) \leq \epsilon ,\,\inProd{ \bm{x}}{ \bm{z} } \leq \epsilon \implies \dist(\bm{x}, \mathcal{K} \cap \{\bm{z}\}^{\perp }) \leq \psi _{\mathcal{K}, \bm{z}} (\epsilon, \| \bm{x} \|).
	  		\]
	  	\end{enumerate}
	  	Then, \( \psi _{\mathcal{K}, \bm{z}} \) is said to be a one-step facial residual function (\( \mathds{1} \)-FRF) for \( \mathcal{K} \) and \( \bm{z} \).
	  \end{definition}
	  The basic idea of the  aforementioned framework  is as follows. Suppose that in each step of the facial reduction algorithm we can find a suitable one-step facial residual function for the ``current'' face and ``next'' exposing vector until we reach a face $\stdFace$ such that $\stdFace$ and $\subSpace + \bm{a}$ satisfy the \textit{partial polyhedral Slater's (PPS) condition} \cite[Definition 3]{L17}.\footnote{ We note that this implies a Lipschitzian error bound holds for $\stdFace$ and $\subSpace + \bm{a}$, see \cite[Corollary~3]{BBL99} and the discussion preceding \cite[Proposition~2.3]{LiLoPo20}.}
	  Then, we can  construct an error bound for $\stdCone$ and $\subSpace + \bm{a}$ by composing these residual functions in a specific manner.
	  In this regard, if \eqref{eq:conic-feasibility-problem} is feasible, we define the \textit{distance to the PPS condition} of \eqref{eq:conic-feasibility-problem}, denoted by $ \dpps(\stdCone, \subSpace + \bm{a}) $, as the length \textit{minus one} of the \textit{shortest} chain of faces (among those chains constructed as in \cite[Proposition~5]{L17}) such that the PPS condition holds for the final face in the chain and $\subSpace + \bm{a}$.
	}

	{
	  We end this section with the following lemma, which is useful in the analysis of one-dimensional faces. It will be used repeatedly in our subsequent discussions.
	  \begin{lemma}[{\cite[Lemma 2.5]{LiLoPo21}}]
		\label{lemma:Lemma-2.11}
		Let \( \stdCone \) be a pointed closed convex cone and let \( \bm{z}\in \partial \stdCone^*\setminus \{{\bm 0}\} \) be such that \( \stdFace := \{\bm{z}\}^{\perp} \cap \stdCone \) is a one-dimensional proper face of \( \stdCone \).
		Let \( \bm{f} \in \stdCone \setminus \{\bm{0}\} \) be such that \( \stdFace = \{t \bm{f} \, | \, t \geq 0\} \).
		Let \( \eta > 0 \) and \( \bm{v} \in \partial \stdCone \cap B(\eta) \setminus \stdFace \), \( \bm{w} = P_{\{\bm{z}\}^{\perp}}(\bm{v}) \) and \( \bm{u} = P_{\stdFace}(\bm{w}) \) with \( \bm{u} \neq \bm{w} \).
		Then it holds that \( \left\langle \bm{f}, \bm{z} \right\rangle = 0 \) and we have
		\[
		  \| \bm{v} - \bm{w} \| = \frac{| \left\langle \bm{z}, \bm{v} \right\rangle |}{\| \bm{z} \|}, \qquad \| \bm{u} - \bm{w} \| =
		  \begin{cases}
			\left\| \bm{v} - \frac{\left\langle \bm{z}, \bm{v} \right\rangle}{\| \bm{z} \|^2}\bm{z} - \frac{\left\langle \bm{f}, \bm{v} \right\rangle}{\| \bm{f} \|^2}\bm{f} \right\| & \text{ if } \left\langle \bm{f}, \bm{v} \right\rangle \geq 0, \\
			\,\,\,\left\| \bm{v} - \frac{\left\langle \bm{z}, \bm{v} \right\rangle}{\| \bm{z} \|^2} \bm{z} \right\| & \text{ otherwise }.
		  \end{cases}
		\]
		Moreover, when \( \left\langle \bm{f}, \bm{v} \right\rangle \geq 0 \) (or, equivalently, \( \left\langle \bm{f}, \bm{w} \right\rangle \geq 0 \)), we have \( \bm{u} = P_{\Span \stdFace}(\bm{w}) \).
		On the other hand, if \( \left\langle \bm{f}, \bm{v} \right\rangle < 0 \), we have \( \bm{u} = \bm{0} \).
	  \end{lemma}
	}

	\section{Error bounds for the generalized power cone}\label{sec:error-bound-power-cone}
	We consider the generalized power cone and its dual.
	Let \( m\ge 1 \), \( n\ge 2 \) and \( \bm{\alpha} = (\alpha_1,\ldots,\alpha_n)\in (0,1)^n \) with \( \sum_{i=1}^n\alpha_i = 1 \), the generalized power cone \( \gK \) and its dual \( \dgK \) are given respectively by
	\begin{equation}\label{eq:gpowercones}
	\begin{aligned}
	\gK  & = \left\{\bm{x} = (\bbx,\tbx)\in \R^{m+n}\,\bigg| \,\|\bbx\|\le \prod_{i=1}^n\tx_i^{\alpha_i},\,\bbx\in \R^m,\,\tbx\in \R^n_+\right\},                               \\
	\dgK & = \left\{\bm{z} = (\bbz,\tbz)\in \R^{m+n}\,\bigg| \,\|\bbz\|\le \prod_{i=1}^n\left(\frac{\tz_i}{\alpha_i}\right)^{\alpha_i},\,\bbz\in \R^m,\,\tbz\in \R^n_+\right\}.
	\end{aligned}
	\end{equation}
	Here, given a vector \( \bm{x}\in \R^{m+n} \), we let \( \bbx\in \R^m \) be the vector corresponding to its first \( m \) entries and \( \tbx\in \R^n \) be the vector corresponding to its last \( n \) entries.

	In this section, we will prove the main result of our paper: a complete analysis of the error bounds of $\gK$. This will require an analysis of the facial structure of $\gK$ which we will do shortly {after the following lemmas}.
	\begin{lemma}\label{lemma:inequ-inprod-power-cone}
		Let \( n\ge 2 \) and \( \bm{\alpha} = (\alpha_1,\ldots,\alpha_n)\in (0,1)^n \) with \( \sum_{i=1}^n\alpha_i = 1 \). Let \( \bm{\zeta}\in {\rm int}\,\R_{-}^n \) satisfy \( \prod_{i=1}^n (-\zeta_i/\alpha_i)^{\alpha_i} = 1 \). Define \( \widetilde{\bm{\zeta}} := -\bm\alpha\circ \bm{\zeta}^{-1} \), where \( \circ \) is the Hadamard product and the inverse is taken componentwise. Then there exist \( C > 0 \) and \( \epsilon > 0 \) so that
		\begin{equation}\label{eq:inequ-inprod-power-cone}
			-1 - \inProd{\bm{\zeta}}{\bm{\omega}} \ge C\|\bm{\omega} - \widetilde{\bm{\zeta}}\|^2 \,\,{whenever}\,\bm{\omega}\in {\rm int}\,\R^n_+,\,\|\bm{\omega} - \widetilde{\bm{\zeta}}\|\le \epsilon {\,and\,} \prod_{i=1}^n\omega_i^{\alpha_i} = 1.
		\end{equation}
		Moreover, for any \( \bm{\omega}\in {\rm int}\,\R_+^n \) satisfying \( \prod_{i=1}^n\omega_i^{\alpha_i} = 1 \), it holds that \( \inProd{\bm{\zeta}}{\bm{\omega}}\le-1 \); furthermore, we have \( \inProd{\bm{\zeta}}{\bm{\omega}}=-1 \) if and only if \( \bm{\omega} = \widetilde{\bm{\zeta}} \).
	\end{lemma}
	\begin{proof}
		For each \( i \), we see from the Taylor series of \( \ln(\cdot) \) at \( \widetilde{\zeta}_i > 0 \) that
		\begin{equation*}
			\ln(\omega_i) = \ln(\widetilde{\zeta}_i) + \widetilde{\zeta}_i^{-1}(\omega_i - \widetilde{\zeta}_i) - \widetilde{\zeta}_i^{-2}(\omega_i - \widetilde{\zeta}_i)^2 + O(|\omega_i - \widetilde{\zeta}_i|^3)\,\,{\rm as}\,\,\omega_i \to \widetilde{\zeta}_i, \,\,\omega_i > 0.
		\end{equation*}
		Thus, there exist \( \epsilon_i > 0 \) and \( c_i > 0 \) so that
		\[
		\ln(\widetilde{\zeta}_i) \ge \ln(\omega_i) - \widetilde{\zeta}_i^{-1}(\omega_i - \widetilde{\zeta}_i) + c_i (\omega_i - \widetilde{\zeta}_i)^2\,\,\,\mbox{whenever} \,|\omega_i - \widetilde{\zeta}_i|\le \epsilon_i \,{\rm and}\,\omega_i > 0.
		\]
		Let \( \epsilon := \min\limits_{1\le i\le n}\epsilon_i > 0 \). Multiplying both sides of the above inequality by \( \alpha_i \) and summing the resulting inequalities from \( i = 1 \) to \( n \), we see that whenever \( \bm{\omega}\in {\rm int}\,\R^n_+ \) satisfies \( \|\bm{\omega} - \widetilde{\bm{\zeta}}\|\le \epsilon \) and \( \prod_{i=1}^n\omega_i^{\alpha_i} = 1 \), we have
		\begin{align}
			&0  \overset{\rm (a)}= \sum_{i=1}^n\alpha_i\ln(\widetilde{\zeta}_i)\ge \sum_{i=1}^n\alpha_i\ln(\omega_i) - \sum_{i=1}^n\alpha_i\widetilde{\zeta}_i^{-1}(\omega_i - \widetilde{\zeta}_i) + \sum_{i=1}^n\alpha_ic_i (\omega_i - \widetilde{\zeta}_i)^2 \notag\\
			& \overset{\rm (b)}= - \sum_{i=1}^n\alpha_i\widetilde{\zeta}_i^{-1}(\omega_i - \widetilde{\zeta}_i) + \sum_{i=1}^n\alpha_ic_i (\omega_i - \widetilde{\zeta}_i)^2
			\overset{\rm (c)}= - \sum_{i=1}^n\alpha_i\widetilde{\zeta}_i^{-1}\omega_i \!+\! 1 \!+\! \sum_{i=1}^n\alpha_ic_i (\omega_i - \widetilde{\zeta}_i)^2                                                                                                                          \notag\\
			& = \sum_{i=1}^n\zeta_i\omega_i + 1 + \sum_{i=1}^n\alpha_ic_i (\omega_i - \widetilde{\zeta}_i)^2,\notag
		\end{align}
		where (a) and (b) hold because \( \prod_{i=1}^n\widetilde{\zeta}_i^{\alpha_i} = \prod_{i=1}^n\omega_i^{\alpha_i} = 1 \), (c) uses the fact that \( \sum_{i=1}^n\alpha_i = 1 \), and the last equality follows from the definition of \( \widetilde{\bm{\zeta}} \). Rearranging the above inequality, we conclude that \eqref{eq:inequ-inprod-power-cone} holds with \( C = \min\limits_{1\le i\le n}\alpha_ic_i > 0 \).

		Next, let \( \bm{\omega}\in {\rm int}\,\R_+^n \) satisfy \( \prod_{i=1}^n\omega_i^{\alpha_i} = 1 \). Then \( (-1,\bm{\omega})\in \K \). Recall from the assumption that \( (1,-\bm{\zeta})\in \dK \). From these we deduce \( \inProd{\bm{\zeta}}{\bm{\omega}}\le-1 \). If \( \inProd{\bm{\zeta}}{\bm{\omega}}=-1 \), then
		\[
		\sum_{i=1}^n \alpha_i\left(\frac{-\zeta_i}{\alpha_i}\right)\omega_i = \sum_{i=1}^n (-\zeta_i)\omega_i = 1 = \prod_{i=1}^n\omega_i^{\alpha_i} = \prod_{i=1}^n\left(\frac{-\zeta_i}{\alpha_i}\right)^{\alpha_i}\prod_{i=1}^n\omega_i^{\alpha_i} .
		\]
		Taking \( \ln \) on both sides of the above equality, we see that
		\[
		\ln\left[\sum_{i=1}^n \alpha_i\left(\frac{-\zeta_i}{\alpha_i}\right)\omega_i\right] = \sum_{i=1}^n\alpha_i \ln\left[\left(\frac{-\zeta_i}{\alpha_i}\right)\omega_i\right].
		\]
		Since \( \ln \) is strictly concave and \( \alpha_i \in (0,1) \) for all \( i \), we conclude that there exists \( c > 0 \) so that \( \omega_i\cdot(-\zeta_i/\alpha_i) = c \) for all \( i \). This, together with the facts that
		\( \prod_{i=1}^n\omega_i^{\alpha_i} = \prod_{i=1}^n (-\zeta_i/\alpha_i)^{\alpha_i} = 1 \) and \( \sum_{i=1}^n\alpha_i = 1 \), gives \( c = 1 \). It thus follows that \( \bm{\omega} = \widetilde{\bm{\zeta}} \). Conversely, it is routine to check that if \( \bm{\omega} = \widetilde{\bm{\zeta}} \), then \( \prod_{i=1}^n\omega_i^{\alpha_i} = 1 \) and \( \inProd{\bm{\zeta}}{\bm{\omega}}=-1 \).
	\end{proof}

	{
	  The next lemma is obtained by applying \cite[Lemma 4.1]{LiLoPo21} with \( p = q = 2 \).
	  \begin{lemma}
		\label{lemma:Lemma-4.1}
		Let \( \bm{\zeta} \in \R^n \) (\( n \geq 1 \)) satisfy \( \| \bm{\zeta} \| = 1 \).
		Define \( \overline{\bm{\zeta}} := -\bm{\zeta} \). Then there exist \( C > 0 \) and \( \epsilon > 0 \) so that
		\begin{equation}
		  \label{eq:inequ-p=q=2}
		  1 + \left\langle \bm{\zeta}, \bm{w} \right\rangle \geq C \sum_{i \in I} | w_i - \overline{\zeta}_i |^2 + \frac{1}{2} \sum_{i \notin I} | w_i |^2  \ \text{ whenever } \| \bm{w} - \overline{\bm{\zeta}} \| \leq \epsilon \ \text{ and } \ \| \bm{w} \| = 1,
		\end{equation}
		where \( I = \{i \,|\, \overline{\zeta}_i \neq 0\} \).
		Furthermore, for any \( \bm{w} \) satisfying \( \| \bm{w} \| \leq  1\), it holds that \( \left\langle \bm{\zeta}, \bm{w} \right\rangle \geq -1 \), with the equality holding if and only if \( \bm{w} = \overline{\bm{\zeta}} \).
	  \end{lemma}
	}

	\subsection[Proper nontrivial exposed faces of power cone]{The facial structure of \( \gK \)}\label{subsec:facial-structure}
	In this subsection, we discuss the faces of $ \gK$.
	{
	  We first characterize the proper nontrivial exposed faces of $\gK$ in the following proposition.
	  \begin{proposition}[Proper nontrivial exposed faces of \( \gK \)]\label{prop:facial-stru-gK}
		Let \( \bm{z} = (\bbz,\tbz)\in \bd\dgK\backslash\{\bm{0}\} \).
		\begin{enumerate}[label={(\roman*)}]
		  \item \label{item:case-I-facial-structure} If \( \bbz \neq \bm{0} \), then \( \bm{z} \) exposes the following one-dimensional face: 
				\begin{equation}\label{eq:1st-type-face-ray}
				  \displayindent0pt
				  \displaywidth\textwidth
				  \Fr := \{\bm{z}\}^\perp\cap \gK =  \{t \bm{f}\in \R^{m+n}\,| \,t\ge 0\}\,\,\mbox{with} \,\bm{f}= (-\bbz/\|\bbz\|^2,\bm{\alpha}\circ \tbz^{-1}),
				\end{equation}
				where the inverse is taken componentwise.

		  \item \label{item:case-II-facial-structure} If \( \bbz = \bm{0} \), then \( \bm{z} \) exposes the following face of dimension \( n - | \stdIndex | \):
				\begin{equation}\label{eq:2nd-type-face-d-dim}
				  \displayindent0pt
				  \displaywidth\textwidth
				  \Fhatz := \{\bm{z}\}^\perp \cap \gK = \{\bm{x} = (\bbx,\tbx)\in \R_+^{m+n}\,| \,\bbx = \bm{0},\,\tx_i = 0\ \forall i\in \stdIndex\},
				\end{equation}
				where \( \stdIndex := \{i \,|\, \tz_i > 0\}\neq \emptyset \) and \( | \stdIndex | \) denotes the cardinality of $\stdIndex$.
		\end{enumerate}
	  \end{proposition}
	}

	\begin{proof}
	  \ref{item:case-I-facial-structure}: Notice that \( \bm{x} = (\bbx,\tbx)\in \{\bm{z}\}^\perp\cap \gK\backslash \{\bm{0}\} \) if and only if \( \bm{x}\in \bd\gK \), \( \bm{x} \neq \bm{0} \) and 
	  \begin{equation}\label{eq:inprod=0-case-I}
		\inProd{\bbz}{\bbx} + \inProd{ \tbz}{\tbx} = 0.
	  \end{equation}
	  The above relation yields
	  \begin{equation}\label{eq:inequ-sum-prod-power-case-I}
		\sum_{i=1}^n \tz_i\tx_i = -\inProd{\bbz}{\bbx}\le \|\bbz\| \|\bbx\| \le \prod_{i=1}^n\left(\frac{\tx_i\tz_i}{\alpha_i}\right)^{\alpha_i},
	  \end{equation}
	  where the last inequality follows from the definition of $\gK$ in \eqref{eq:gpowercones}.

	  Note that \( \tx_i \) cannot be \textbf{all} zero, for otherwise \( \bbx \) will also be zero since \( \bm{x}\in \bd\gK \), which contradicts \( \bm{x}\neq \bm{0} \). In addition, we must have \( \tz_i > 0 \) for all \( i \) because \( \bbz \neq \bm{0} \) and \( \bm{z}\in \bd\dgK\backslash\{\bm{0}\} \). Using these observations, we have $\sum_{i=1}^n \tz_i\tx_i >0$. Combining this with \eqref{eq:inequ-sum-prod-power-case-I}, we deduce that \( \tx_i\tz_i > 0 \) for all \( i \). Now we can take \( \ln \) on both sides of \eqref{eq:inequ-sum-prod-power-case-I} to obtain
	  \begin{equation}\label{eq:log-inequ-sum-prod-power-case-I}
		\ln\left[\sum_{i=1}^n \alpha_i\left(\frac{\tx_i\tz_i}{\alpha_i}\right)\right] \le \alpha_1\ln\left(\frac{\tx_1\tz_1}{\alpha_1}\right) + \cdots + \alpha_n\ln\left(\frac{\tx_n\tz_n}{\alpha_n}\right).
	  \end{equation}
	  Using this together with the fact that \( \ln(\cdot) \) is strictly concave, we deduce that \eqref{eq:log-inequ-sum-prod-power-case-I} holds as an equality. Hence, there exists a constant \( c > 0 \) so that
	  \begin{equation}\label{eq:form-tildex-by-tildez-case-I}
		\tx_i = c \alpha_i\tz_i^{-1}\,\,\,\forall i = 1, 2, \dots , n.
	  \end{equation}
	  Plugging \eqref{eq:form-tildex-by-tildez-case-I} into \eqref{eq:inprod=0-case-I}, we obtain
	  \begin{equation}\label{eq:inprod-barz-barx=c-case-I}
		\inProd{\bbz}{\bbx} = -\inProd{ \tbz}{\tbx} = -c\sum_{i=1}^n \alpha_i = -c.
	  \end{equation}
	  Moreover, using \eqref{eq:form-tildex-by-tildez-case-I} and the last relation in \eqref{eq:inequ-sum-prod-power-case-I}, we see that
	  \[
		\|\bbz\| \|\bbx\| \le \prod_{i=1}^n\left(\frac{\tx_i\tz_i}{\alpha_i}\right)^{\alpha_i} = c.
	  \]
	  {The two displayed lines} above show that \( \|\bbz\| \|\bbx\| = -\inProd{\bbz}{\bbx} \), which together with \( \bbz \neq \bm{0} \) implies that there exists \( \kappa > 0 \) so that
	  \begin{equation}\label{eq:form-barx-by-barz-kappa}
		\bbx = -\kappa\bbz.
	  \end{equation}
	  Plugging \eqref{eq:form-barx-by-barz-kappa} into \eqref{eq:inprod-barz-barx=c-case-I}, we obtain that \( \kappa = c/\|\bbz\|^2 \). Using this together with \eqref{eq:form-tildex-by-tildez-case-I} and \eqref{eq:form-barx-by-barz-kappa}, we can now conclude that
	  \[
		\Fr := \{\bm{z}\}^\perp\cap \gK =  \{t \bm{f}\in \R^{m+n}\,| \,t\ge 0\}\,\,\mbox{with} \,\bm{f}= (-\bbz/\|\bbz\|^2,\bm{\alpha}\circ \tbz^{-1}),
	  \]
	  where the inverse is taken componentwise.

	  \ref{item:case-II-facial-structure}: In this case, \( \bbz = \bm{0} \). Then \( \stdIndex:= \{i\,| \,\tz_i > 0\} \) is nonempty because \( \bm{z}\neq \bm{0} \). Hence, \( \bm{x} = (\bbx,\tbx)\in \{\bm{z}\}^\perp\cap \gK\backslash \{\bm{0}\} \) if and only if \( \bm{x}\in \bd\gK\backslash\{\bm{0}\} \) and satisfies
	  \begin{equation*}
		\sum_{i\in \stdIndex} \tz_i\tx_i = 0.
	  \end{equation*}
	  This means that \( \tx_i = 0 \) whenever \( i\in \stdIndex \) and hence \( \bbx= \bm{0} \). Thus,
	  \[
		\Fhatz := \{\bm{z}\}^\perp \cap \gK = \{\bm{x} = (\bbx,\tbx)\in \R_+^{m+n}\,| \,\bbx = \bm{0},\,\tx_i = 0\ \forall i\in \stdIndex\}.
	  \]
	\end{proof}

	Having characterized the proper exposed faces of $\gK$, we will show that $\gK$ is \emph{projectionally exposed} \cite{BW81,ST90}, which means that for every face $\stdFace$ of $\gK$ there is a linear operator ${\bm P}$ satisfying ${\bm P}(\gK) = \stdFace$ and ${\bm P}^2 = {\bm P}$. In particular, ${\bm P}$, which depends on $\stdFace$, is a projection that is not necessarily orthogonal.
Projectionally exposed cones are both facially exposed \cite[Corollary~4.4]{ST90} and amenable \cite[Proposition~9]{L17}, see also \cite{LRS20}.
	\begin{proposition}[Generalized power cones are projectionally exposed]
	$\gK$ is projectionally exposed, in particular, all its faces are exposed. 
	\end{proposition}
	\begin{proof}
	Sung and Tam proved in \cite[Corollary~4.5]{ST90} that a sufficient condition for a cone to be projectionally exposed is that all its exposed faces are projectionally exposed.
	With this in mind, let $\stdFace$ be an exposed face of $\gK$. If $\stdFace = \{\bm{0}\}$ or $\stdFace = \gK$, then the zero map and the identity map are, respectively, projections mapping $\gK$ to $\stdFace$.
	Otherwise, $\stdFace$ is a nonzero proper face of $\gK$ and is of the form $\{\bm{z}\}^\perp \cap \gK $, for some \( \bm{z} = (\bbz,\tbz)\in \bd\dgK\backslash\{\bm{0}\}\).
	By the analysis in cases \ref{item:case-I-facial-structure}, \ref{item:case-II-facial-structure}, we only need to consider two cases.
	
	First, suppose that $\stdFace$ is a one-dimensional face as in \eqref{eq:1st-type-face-ray} and let $\bm{u} \in \dgK$ be such that $	\inProd{\bm{f}}{\bm{u}} = 1$. At least one such $\bm{u}$ exists, since otherwise we would have $\bm{f} \in (\dgK)^\perp = \{\bm{0}\}$. Then, ${\bm P} = \bm{f} \bm{u}^\top$ satisfies ${\bm P}^2 = {\bm P}$ and ${\bm P}(\gK) = \stdFace$ as required.
	
	Next, suppose that $\stdFace$ is as in \eqref{eq:2nd-type-face-d-dim}. Then, we let ${\bm P}$ be the linear map that maps $(\bbx, \tbx)$ to $(\bm{0},\tby)$ where $\ty_i = 0$ if $i \in \stdIndex$ and $\ty_i = \tx_i$ if $i \not \in \stdIndex$. With that, ${\bm P}$ is a projection mapping $\gK$ to $\stdFace$.
	\end{proof}

	\subsection[Deducing error bounds and one-step facial residual functions for the power cone]{Deducing error bounds and one-step facial residual functions for \( \gK \)}\label{subsec:error-bound-1-FRF-power-cone}

	We start with the faces \( \Fr \) that correspond to a \( \bm{z}\in \bd\dgK\backslash\{\bm{0}\} \) with \( \bbz \neq \bm{0} \).
    {We have the following result.}
	\begin{theorem}\label{thm:Holderian-1/2-1st-type-face}
		Let \( \bm{z} = (\bbz,\tbz)\in \bd\dgK\backslash\{\bm{0}\} \) with \( \bbz\neq \bm{0} \) and let \( \Fr := \{\bm{z}\}^\perp\cap \gK \). Let \( \eta > 0 \) and define
		\begin{equation}\label{eq:gamma-1/2-1st-type-face}
			\gamma_{\bm{z},\eta}\! :=\! \inf_{\bm{v}}\left\{\frac{\|\bm{v} - \bm{w}\|^\frac12}{\|\bm{u} - \bm{w}\|}\,\bigg| \,
\begin{array}{c}
\bm{v}\in \bd \gK\cap \B(\eta)\backslash \Fr,\,\bm{w} = P_{\{\bm{z}\}^\perp}(\bm{v}),\\
\bm{u} = P_{\Fr}(\bm{w}),\,\bm{u} \neq \bm{w}
\end{array}
\right\}.
		\end{equation}
		Then it holds that \( \gamma_{\bm{z},\eta}\in (0,\infty] \) and that
		\[
		\dist(\bm{q},\Fr)\le \max\{2\sqrt{\eta},2\gamma_{\bm{z},\eta}^{-1}\}\cdot\dist(\bm{q},\gK)^{\frac{1}{2}}\,\,\text{ whenever }  \bm{q} \in \{\bm{z}\}^\perp\cap \B(\eta) .
		\]
	\end{theorem}

	\begin{proof}
		Suppose for a contradiction that \( \gamma_{\bm{z},\eta} = 0 \). Then, {in view of \cite[Lemma 3.12]{LiLoPo20}}, there exist \( \widehat{\bm{v}}\in \Fr \) and a sequence \( \{\bm{v}^k\}\subset \bd \gK\cap \B(\eta)\backslash \Fr \) such that
		\begin{equation}\label{eq:forcontradiction-1st-face}
			\lim_{k\to \infty}\bm{v}^k = \lim_{k\to \infty}\bm{w}^k = \widehat{\bm{v}}\,\,{\rm and}\,\,\lim_{k\to\infty}\frac{\|\bm{w}^k - \bm{v}^k\|^{\frac{1}{2}}}{\|\bm{w}^k - \bm{u}^k\|} = 0,
		\end{equation}
		where \( \bm{w}^k = P_{\{\bm{z}\}^\perp}(\bm{v}^k) \), \( \bm{u}^k = P_{\Fr}(\bm{w}^k) \) and \( \bm{u}^k\neq \bm{w}^k \).

		Define, for notational simplicity, \( z_0 := \|\bbz\| \) and \( v^k_0:=\|\bbv^k\| \).
		Then, since \( \{\bm{v}^k\}\subset \bd \gK \) and \( \bm{z} \in \bd\dgK \) with \( \bbz \neq \bm{0} \), we have\vspace{-0.15 cm}
		\begin{equation}\label{eq:z0-v0-notation}\vspace{-0.15 cm}
			z_0 = \|\bbz\| = \prod_{i=1}^n\left(\frac{\tz_i}{\alpha_i}\right)^{\alpha_i}> 0 \,\,{\rm and}\,\,v^k_0=\|\bbv^k\| = \prod_{i=1}^n(\tv_i^k)^{\alpha_i}\,\,\forall k.
		\end{equation}

		If it holds that \( v_0^k = 0 \) infinitely often, by passing to a further subsequence, we may assume that \( v_0^k = 0 \) for all \( k \). Then we have in view of {Lemma \ref{lemma:Lemma-2.11}} that\vspace{-0.15 cm}
		\[\vspace{-0.15 cm}
		\!\|\bm{v}^k - \bm{w}^k\| \!=\! \frac1{\|\bm{z}\|}|\inProd{ \tbz}{\tbv^k}|\!\overset{\rm (a)}=\! \frac1{\|\bm{z}\|}\sum_{i=1}^n\tz_i\tv^k_i\!\ge\! \frac{\min_i \tz_i}{\|\bm{z}\|} \|\tbv^k\|_1\!\ge\! \frac{\min_i \tz_i}{\|\bm{z}\|} \|\tbv^k\|
		\!\overset{\rm (b)}=\!\frac{\min_i \tz_i}{\|\bm{z}\|} \|\bm{v}^k\|,
		\]
		where (a) holds because \( \tv^k_i \ge 0 \) and \( \tz_i > 0 \) for all \( i \) (see \eqref{eq:z0-v0-notation}), and (b) holds since \( \|\bbv^k\| = 0 \). Since
		\( \|\bm{w}^k - \bm{u}^k\| = \dist(\bm{w}^k,\Fr) \le \|\bm{w}^k\|\le \|\bm{v}^k\| \) as a consequence of the properties of projections, we conclude from this and the above display that
		\( \|\bm{v}^k - \bm{w}^k\| \ge \frac{\min_i \tz_i}{\|\bm{z}\|}\|\bm{w}^k - \bm{u}^k\| \), contradicting \eqref{eq:forcontradiction-1st-face}.

		Thus, by considering a further subsequence if necessary, from now on, we assume\vspace{-0.15 cm}
		\begin{equation}\label{eq:vk0-notation-nonzero}\vspace{-0.15 cm}
			v^k_0=\|\bbv^k\| = \prod_{i=1}^n(\tv_i^k)^{\alpha_i} > 0 \,\,\forall k.
		\end{equation}
		Using {Lemma \ref{lemma:Lemma-2.11}}, we see that\vspace{-0.15 cm}
		\begin{equation}\label{eq:dist-vk-wk}\vspace{-0.15 cm}
			\begin{aligned}
				\|\bm{v}^k - \bm{w}^k\| & =\frac1{\|\bm{z}\|}|\inProd{ \bm{z}}{\bm{v}^k}| = \frac1{\|\bm{z}\|}\left|\sum_{i=1}^{m}\bz_i\bv^k_i + \sum_{i=1}^n\tz_i\tv^k_i\right|             \\
				& = \frac1{\|\bm{z}\|}\left|z_0v_0^k + \sum_{i=1}^{m}\bz_i\bv^k_i - \sum_{i=1}^n(-\tz_i)\tv^k_i - z_0v_0^k\right|                                        \\
				& = \frac{z_0v_0^k}{\|\bm{z}\|}\left| 1 + \inProd{ z_0^{-1}\bbz}{ (v_0^k)^{-1}\bbv^k} - \inProd{ z_0^{-1}(-\tbz)}{ (v_0^k)^{-1}\tbv^k} - 1\right|  \\
				& = \frac{z_0}{\|\bm{z}\|}\left( 1 + \inProd{ z_0^{-1}\bbz}{ (v_0^k)^{-1}\bbv^k} - \inProd{ z_0^{-1}(-\tbz)}{ (v_0^k)^{-1}\tbv^k} - 1\right)v_0^k,
			\end{aligned}
		\end{equation}
		where the last equality holds as \( \|z_0^{-1}\bbz\|= 1 \), \( \|(v_0^k)^{-1}\bbv^k\|= 1 \) and  \( \inProd{ z_0^{-1}\tbz}{ (v_0^k)^{-1}\tbv^k} \ge 1 \), thanks to \eqref{eq:z0-v0-notation}, \eqref{eq:vk0-notation-nonzero} and Lemma~\ref{lemma:inequ-inprod-power-cone} applied with \( \bm{\zeta} = -z_0^{-1}\tbz \).

		Let \( \bm{f} \) be defined as in \eqref{eq:1st-type-face-ray}. We consider two cases:
		\begin{enumerate}[label={(\Roman*)}]
			\item \( \inProd{ \bm{f}}{ \bm{v}^k}\ge 0 \) for all sufficiently large \( k \).
			\item \( \inProd{ \bm{f}}{ \bm{v}^k}< 0 \) infinitely often.
		\end{enumerate}

		(I): By passing to a further subsequence, we may assume that \( \inProd{ \bm{f}}{ \bm{v}^k}\ge 0 \) for all \( k \). In this case, if we define
		\[
		\bm{Q} = \bm{I}_{m + n} - \frac{\bm{z}\bm{z}^{\top}}{\|\bm{z}\|^2} - \frac{\bm{f}\bm{f}^{\top}}{\|\bm{f}\|^2},
		\]
		where \( \bm{f} \) is as in \eqref{eq:1st-type-face-ray}, then we see from {Lemma \ref{lemma:Lemma-2.11}} and \eqref{eq:z0-v0-notation} that
		\begin{equation}\label{eq:upbd-dist-uk-wk-two-parts}
			\begin{aligned}
				\|\bm{u}^k - \bm{w}^k\| & = \|\bm{Q}\bm{v}^k\| = v_0^k\left\|\bm{Q}\begin{bmatrix}
					(v_0^k)^{-1}\bbv^k \\ (v_0^k)^{-1}\tbv^k
				\end{bmatrix}\right\| \\
&\overset{\rm (a)}= v_0^k{\Bigg\|}\bm{Q}\begin{bmatrix}
					(v_0^k)^{-1}\bbv^k \\ (v_0^k)^{-1}\tbv^k
				\end{bmatrix} \!-\! \bm{Q}\underbrace{\begin{bmatrix}
						-z_0^{-1}\bbz \\ \bm{\alpha}\circ(z_0\tbz^{-1})
				\end{bmatrix}}_{z_0\bm{f}}{\Bigg\|} \\
				& \le v_0^k\left[\|(v_0^k)^{-1}\bbv^k + z_0^{-1}\bbz\| + \|(v_0^k)^{-1}\tbv^k - \bm{\alpha}\circ(z_0\tbz^{-1})\|\right],
			\end{aligned}
		\end{equation}
		where (a) holds because 
		\( \bm{Q}\bm{f} = \bm{0} \) (an identity which is clear from the definitions).

		Next, in view of \eqref{eq:z0-v0-notation}, we can apply {Lemma \ref{lemma:Lemma-4.1}} to obtain \( C_1 > 0 \) and \( \epsilon_1>0 \) so that {\eqref{eq:inequ-p=q=2}} holds with \( \overline{\bm{\zeta}} = -z_0^{-1}\bbz \), i.e.,
		\[
			1 + \inProd{ z_0^{-1}\bbz}{ \bm{\omega}}\ge C_1\|\bm{\omega}  + z_0^{-1}\bbz\|^2
		\]
		whenever \( \|\bm{\omega}  + z_0^{-1}\bbz\|\le \epsilon_1 \) and \( \|\bm{\omega}\|=1 \). On the other hand, in view of the positivity of $1 + \inProd{ z_0^{-1}\bbz}{ \bm{\omega}}$ when $\|\bm{\omega}\|=1$ and $\bm{\omega} \neq -z_0^{-1}\bbz$ {(see Lemma \ref{lemma:Lemma-4.1})}, we know that
\[
C_2 := \inf_{\|\bm{\omega}\|=1}\{1 + \inProd{ z_0^{-1}\bbz}{ \bm{\omega}} \,|\, \|\bm{\omega}  + z_0^{-1}\bbz\|\ge \epsilon_1 \} > 0.
\]
This together with the fact $\|z_0^{-1}\bbz\|=1$ (see \eqref{eq:z0-v0-notation}) implies that
		\[
			1 + \inProd{ z_0^{-1}\bbz}{ \bm{\omega}}\ge C_2 \ge 0.25 C_2 \|\bm{\omega}  + z_0^{-1}\bbz\|^2,
		\]
		whenever \( \|\bm{\omega}  + z_0^{-1}\bbz\|\ge \epsilon_1 \) and \( \|\bm{\omega}\|=1 \). We thus have (with $C_3 \!:=\! \min\{C_1,C_2/4\}$)
\begin{equation}\label{eq:keylemmarel}
			1 + \inProd{ z_0^{-1}\bbz}{ \bm{\omega}}\ge C_3\|\bm{\omega}  + z_0^{-1}\bbz\|^2\ \ \ {\rm whenever}\ \ \|\bm{\omega}\|=1.
		\end{equation}

		In addition, noting \eqref{eq:z0-v0-notation} again, we can apply Lemma~\ref{lemma:inequ-inprod-power-cone} with \( \bm{\zeta} = -z_0^{-1}\tbz\in {\rm int}\,\R^n_- \) to obtain \( C_4 > 0 \) and \( \epsilon > 0 \) so that \eqref{eq:inequ-inprod-power-cone} holds with \( \widetilde{\bm{\zeta}} = \bm{\alpha}\circ (z_0\tbz^{-1}) \), i.e.,
		\begin{equation}\label{eq:keylemmare2}
			-1 + \inProd{ z_0^{-1}\tbz}{ \bm{\omega}}\ge C_4\|\bm{\omega}  - \bm{\alpha}\circ(z_0\tbz^{-1})\|^2
		\end{equation}
		whenever \( \|\bm{\omega}  - \bm{\alpha}\circ(z_0\tbz^{-1})\|\le \epsilon \), \( \bm{\omega}\in {\rm int}\,\R^n_+ \) and \( \prod_{i=1}^n\omega_i^{\alpha_i}=1 \).


Furthermore, consider \( h:\R^n\to \R\cup\{\infty\} \) defined by
		\begin{equation}\label{eq:def-h}
			h(\bm{\omega}) = \begin{cases}
				\displaystyle\frac{\inProd{ z_0^{-1}\tbz}{ \bm{\omega}}-1}{\|\bm{\omega} - \bm{\alpha}\circ(z_0\tbz^{-1})\|} & {\rm if}\,\|\bm{\omega}-\bm{\alpha}\circ(z_0\tbz^{-1})\|\ge \epsilon,\,\bm{\omega}\in \Upsilon, \\
				\infty                                                                                                    & {\rm otherwise},
			\end{cases}
		\end{equation}
		where \( \Upsilon= \{\bm{\omega}\in \R^n_+\,| \,\prod_{i=1}^n \omega_i^{\alpha_i}=1\} \). Then we have
		\[
		\begin{aligned}
			& \liminf_{\|\bm{\omega}\|\to \infty}h(\bm{\omega}) = \liminf_{\|\bm{\omega}\|\to \infty, \bm{\omega}\in \Upsilon}\frac{\inProd{ z_0^{-1}\tbz}{ \bm{\omega}}-1}{\|\bm{\omega} - \bm{\alpha}\circ(z_0\tbz^{-1})\|} \\
			& \ge \liminf_{\|\bm{\omega}\|\to \infty,\bm{\omega}\in \R^n_+}\frac{\inProd{ z_0^{-1}\tbz}{ \bm{\omega}}-1}{\|\bm{\omega} - \bm{\alpha}\circ(z_0\tbz^{-1})\|}
			\overset{\rm (a)}{\ge} \inf_{\|\bm{\lambda}\|=1,\bm{\lambda}\in \R^n_+}\inProd{ z_0^{-1}\tbz}{\bm{\lambda}} \overset{\rm (b)}\ge \min_{1\le i\le n}z_0^{-1}\tz_i > 0.
		\end{aligned}
		\]
		Here (a) may be verified by multiplying both numerator and denominator of the left side by $1/\|\bm{\omega}\|$; (b) holds because \( z_0^{-1}\tz_i > 0 \) for all \( i \) (see \eqref{eq:z0-v0-notation}). Since \( h \) in \eqref{eq:def-h} is also lower semicontinuous on any compact set and {\bf is always positive},\footnote{The positivity can be seen by applying Lemma~\ref{lemma:inequ-inprod-power-cone} with $\bm\zeta = -z_0^{-1}\tbz$.} it must then hold that \( C_5:= \inf h > 0 \). In particular, this means that
		\begin{equation}\label{eq:keylemmare4}
			-1 + \inProd{ z_0^{-1}\tbz}{ \bm{\omega}}\ge C_5\|\bm{\omega}  - \bm{\alpha}\circ(z_0\tbz^{-1})\|
		\end{equation}
		whenever \( \|\bm{\omega}  - \bm{\alpha}\circ(z_0\tbz^{-1})\|\ge \epsilon \), \( \bm{\omega}\in {\rm int}\,\R^n_+ \) and \( \prod_{i=1}^n\omega_i^{\alpha_i}=1 \).

		By passing to suitable subsequences, we will end up with one of the following two cases:
		\begin{enumerate}[label={Case \arabic*:}]
			\item \( \|(v_0^k)^{-1}\tbv^k - \bm{\alpha}\circ(z_0\tbz^{-1})\| \le \epsilon \) for all \( k \). Then we have from \eqref{eq:keylemmarel} (with $\bm{\omega}=(v_0^k)^{-1}\bbv^k$) and \eqref{eq:keylemmare2} (with $\bm{\omega}=(v_0^k)^{-1}\tbv^k$) that for these \( k \)
			\begin{equation*}
				\begin{aligned}
					& 1 + \inProd{ z_0^{-1}\bbz}{ (v_0^k)^{-1}\bbv^k} - \inProd{ z_0^{-1}(-\tbz)}{ (v_0^k)^{-1}\tbv^k} - 1          \\
					& \ge \min\{C_3,C_4\}(\|(v_0^k)^{-1}\bbv^k  + z_0^{-1}\bbz\|^2 + \|(v_0^k)^{-1}\tbv^k - \bm{\alpha}\circ(z_0\tbz^{-1})\|^2).
				\end{aligned}
			\end{equation*}
			Combining this with \eqref{eq:dist-vk-wk} and \eqref{eq:upbd-dist-uk-wk-two-parts}, we see further that
			\[
			\begin{aligned}
				&\|\bm{v}^k - \bm{w}^k\|\\ & \ge \frac{z_0}{\|\bm{z}\|}\min\{C_3,C_4\}(\|(v_0^k)^{-1}\bbv^k  +  z_0^{-1}\bbz\|^2 + \|(v_0^k)^{-1}\tbv^k - \bm{\alpha}\circ(z_0\tbz^{-1})\|^2) v_0^k    \\
				& \ge \frac{z_0}{2\|\bm{z}\|}\min\{C_3,C_4\}(\|(v_0^k)^{-1}\bbv^k  +  z_0^{-1}\bbz\| + \|(v_0^k)^{-1}\tbv^k - \bm{\alpha}\circ(z_0\tbz^{-1})\|)^2 v_0^k     \\
				& \ge \frac{z_0\min\{C_3,C_4\}}{2\|\bm{z}\|v_0^k}\|\bm{u}^k - \bm{w}^k\|^2\ge \frac{z_0\min\{C_3,C_4\}}{2\|\bm{z}\|\eta}\|\bm{u}^k - \bm{w}^k\|^2,
			\end{aligned}
			\]
			where the last inequality holds because \( \bm{v}^k\in \B(\eta) \). The above display contradicts \eqref{eq:forcontradiction-1st-face} and hence Case 1 cannot happen.

			\item \( \|(v_0^k)^{-1}\tbv^k - \bm{\alpha}\circ(z_0\tbz^{-1})\|\ge \epsilon \) for all \( k \). Then we have from \eqref{eq:keylemmarel} and \eqref{eq:keylemmare4} that for these \( k \)
			\begin{equation*}
				\begin{aligned}
					& 1 + \inProd{ z_0^{-1}\bbz}{ (v_0^k)^{-1}\bbv^k} - \inProd{ z_0^{-1}(-\tbz)}{ (v_0^k)^{-1}\tbv^k} - 1        \\
					& \ge \min\{C_3,C_5\}(\|(v_0^k)^{-1}\bbv^k  + z_0^{-1}\bbz\|^2 + \|(v_0^k)^{-1}\tbv^k - \bm{\alpha}\circ(z_0\tbz^{-1})\|).
				\end{aligned}
			\end{equation*}
			Using this together with \eqref{eq:dist-vk-wk}, we deduce that for all large \( k \),
			\[
			\!\|\bm{v}^k - \bm{w}^k\| \!\ge\! \frac{z_0\min\{C_3,C_5\}}{\|\bm{z}\|}(\|(v_0^k)^{-1}\bbv^k  +  z_0^{-1}\bbz\|^2 + \|(v_0^k)^{-1}\tbv^k - \bm{\alpha}\circ(z_0\tbz^{-1})\|) v_0^k.
			\]
			This implies that
			\begin{equation}\label{eq:upbd-part-dist-uk-wk-case-2}
				\begin{aligned}
					\|(v_0^k)^{-1}\bbv^k  +  z_0^{-1}\bbz\|            & \le M_1\sqrt{(v_0^k)^{-1}\|\bm{v}^k - \bm{w}^k\|}, \\
					\|(v_0^k)^{-1}\tbv^k - \bm{\alpha}\circ(z_0\tbz^{-1})\| & \le M_1(v_0^k)^{-1}\|\bm{v}^k - \bm{w}^k\|,
				\end{aligned}
			\end{equation}
			where \( M_1 := \max\left\{\left(\frac{z_0}{\|\bm{z}\|}\min\{C_3,C_5\}\right)^{-1},\left(\frac{z_0}{\|\bm{z}\|}\min\{C_3,C_5\}\right)^{-1/2}\right\} \). Using \eqref{eq:upbd-part-dist-uk-wk-case-2} together with \eqref{eq:upbd-dist-uk-wk-two-parts}, we obtain that
			\begin{equation}
			\begin{aligned}\label{eqn:3Ms}
				\|\bm{u}^k - \bm{w}^k\| & \le M_1v_0^k\left[\sqrt{(v_0^k)^{-1}\|\bm{v}^k - \bm{w}^k\|} + (v_0^k)^{-1}\|\bm{v}^k - \bm{w}^k\|\right]   \\
				&\stackrel{\rm (a)}{\le}M_1\sqrt{\eta} \sqrt{\|\bm{v}^k-\bm{w}^k\|} + M_1\|\bm{v}^k - \bm{w}^k\| \\
				&\stackrel{\rm (b)}{\le} 3M_1\sqrt{\eta}\sqrt{\|\bm{v}^k - \bm{w}^k\|},
			\end{aligned}
			\end{equation}
			where (a) holds since \( \bm{v}^k\in \B(\eta) \) wherefore $v_0^k \leq \eta$, and (b) holds because \(\|\bm{w}^k\| \leq \|\bm{v}^k\| \le \eta \) (because the projection onto $\stdCone$ is nonexpansive and $\bm 0\in \stdCone$), wherefore
			$$
			\|\bm{w}^k-\bm{v}^k\| = \sqrt{\|\bm{w}^k-\bm{v}^k\|}\sqrt{\|\bm{w}^k-\bm{v}^k\|} \leq 2\sqrt{\eta}\sqrt{\|\bm{w}^k-\bm{v}^k\|}.
			$$
			Altogether, \eqref{eqn:3Ms} contradicts \eqref{eq:forcontradiction-1st-face} and hence Case 2 cannot happen.
		\end{enumerate}
		Summarizing the above discussions, we see that Case (I) cannot happen.

		(II):  By passing to a further subsequence, we may assume that \( \inProd{ \bm{f}}{ \bm{v}^k}< 0 \) for all \( k \). This together with the definition of \( \bm{f} \) gives
		\[
		\frac{v_0^k}{z_0}[\inProd{ -z_0^{-1}\bbz}{ (v_0^k)^{-1}\bbv^k} + \inProd{ \bm{\alpha}\circ(z_0\tbz^{-1})}{ (v_0^k)^{-1}\tbv^k}] < 0.
		\]
		Since $\frac{v_0^k}{z_0} > 0$, we deduce that $\inProd{ -z_0^{-1}\bbz}{ (v_0^k)^{-1}\bbv^k} + \inProd{ \bm{\alpha}\circ(z_0\tbz^{-1})}{ (v_0^k)^{-1}\tbv^k} < 0$ for all $k$. Then it must hold that
		\[
		\lim_{k\to \infty}\|(v_0^k)^{-1}\bbv^k + z_0^{-1}\bbz\|+\|(v_0^k)^{-1}\tbv^k - \bm{\alpha}\circ(z_0\tbz^{-1})\|\neq 0;
		\]
otherwise, we have $(v_0^k)^{-1}\bbv^k\to -z_0^{-1}\bbz$ and $(v_0^k)^{-1}\tbv^k\to \bm{\alpha}\circ(z_0\tbz^{-1})$, which further gives $\inProd{ -z_0^{-1}\bbz}{ (v_0^k)^{-1}\bbv^k} + \inProd{ \bm{\alpha}\circ(z_0\tbz^{-1})}{ (v_0^k)^{-1}\tbv^k}\to \|z_0^{-1}\bbz\|^2+\|\bm{\alpha}\circ(z_0\tbz^{-1})\|^2 = \|z_0\bm f\|^2 > 0$, a contradiction.

		Consequently, there exists $\epsilon > 0$ such that for all sufficiently large $k$,
		\begin{equation}\label{eq:sum-out-of-neighbor}
			\|(v_0^k)^{-1}\bbv^k + z_0^{-1}\bbz\|+\|(v_0^k)^{-1}\tbv^k - \bm{\alpha}\circ(z_0\tbz^{-1})\| \ge \epsilon.
		\end{equation}
		Consider the function \( G:\R^{m+n}\to \R\cup\{\infty\} \) defined by
		\[
		G(\bm{\xi},\bm{\omega}) := \begin{cases}
			\displaystyle \frac{|\inProd{ z_0^{-1}\bbz}{ \bm{\xi}} + \inProd{ z_0^{-1}\tbz}{ \bm{\omega}}|}{\sqrt{1+\|\bm{\omega}\|^2}} & {\rm if}\,(\bm{\xi},\bm{\omega})\in \Xi,\,\|\bm{\xi}\|=1,\,{\rm and}\,\bm{\omega}\in \Upsilon, \\
			\infty                                                                                                                       & {\rm otherwise},
		\end{cases}
		\]
		where \( \Upsilon = \{\bm{\omega}\in \R^n_+\,| \,\prod_{i=1}^n \omega_i^{\alpha_i}=1\} \) and \( \Xi= \{(\bm{\xi},\bm{\omega})\,| \,\|\bm{\xi} + z_0^{-1}\bbz\|+\|\bm{\omega} - \bm{\alpha}\circ(z_0\tbz^{-1})\| \ge \epsilon\} \).
		Since \( \inProd{ z_0^{-1}\bbz}{ \bm{\xi}} + \inProd{ z_0^{-1}\tbz}{ \bm{\omega}} = 1+\inProd{ z_0^{-1}\bbz}{ \bm{\xi}} - \inProd{ z_0^{-1}(-\tbz)}{ \bm{\omega}}-1 \), we see from \eqref{eq:z0-v0-notation}, \eqref{eq:sum-out-of-neighbor}, {Lemma \ref{lemma:Lemma-4.1}} and Lemma~\ref{lemma:inequ-inprod-power-cone} that \( G \) is never zero. Moreover, it is clearly lower semicontinuous on any compact set, and
		\[
		\begin{aligned}
			\!\liminf_{\|(\bm{\xi},\bm{\omega})\|\to \infty}G(\bm{\xi},\bm{\omega}) \!=\!\!\! \liminf_{\|\bm{\omega}\|\to \infty,\bm{\omega}\in \Upsilon}\!\!\frac{|\inProd{ z_0^{-1}\tbz}{ \bm{\omega}}|}{\sqrt{1+\|\bm{\omega}\|^2}}\!\overset{\rm (a)}\ge\!\!\! \inf_{\|\bm{\lambda}\|=1,\bm{\lambda}\in \R^{n}_+}\!\!|\inProd{ z_0^{-1}\tbz}{\bm{\lambda}}|\!\overset{\rm (b)}\ge\!\! \min_i |z_0^{-1}\tz_i| > 0,
		\end{aligned}
		\]
		where (a) may be verified by multiplying numerator and denominator by $1/\|\bm{\omega}\|$ and (b) holds since \( z_0^{-1}\tz_i > 0 \) for all \( i \).
		Thus, \( C_6:= \inf G > 0 \) and we have for all large \( k \),
		\[
		\begin{aligned}
			\frac{\|\bm{v}^k - \bm{w}^k\|}{\|\bm{u}^k - \bm{w}^k\|} & \overset{\rm (a)}= \frac{\|\bm{v}^k - \bm{w}^k\|}{\|\bm{w}^k\|} \overset{\rm (b)}\ge \frac{\|\bm{v}^k - \bm{w}^k\|}{\|\bm{v}^k\|} \\
&\overset{\rm (c)}=
			\frac{z_0}{\|\bm{z}\|}\frac{|\inProd{ z_0^{-1}\bbz}{ (v_0^k)^{-1}\bbv^k} + \inProd{ z_0^{-1}\tbz}{ (v_0^k)^{-1}\tbv^k}|v_0^k}{\sqrt{(v^k_0)^2 + \|\tbv^k\|^2}}                                                                                                            \\
			& = \frac{z_0}{\|\bm{z}\|}\frac{|\inProd{ z_0^{-1}\bbz}{ (v_0^k)^{-1}\bbv^k} + \inProd{ z_0^{-1}\tbz}{ (v_0^k)^{-1}\tbv^k}|}{\sqrt{1 + \|(v_0^k)^{-1}\tbv^k\|^2}}\overset{\rm (d)}\ge \frac{C_6 z_0}{\|\bm{z}\|},
		\end{aligned}
		\]
		where (a) follows from {Lemma \ref{lemma:Lemma-2.11}}, which states that \( \bm{u}^k = \bm{0} \) in Case (II), (b) holds because the projection onto the cone is nonexpansive and $\bm 0$ is in the cone, (c) follows from \eqref{eq:dist-vk-wk} and (d) follows from \eqref{eq:sum-out-of-neighbor}, \eqref{eq:vk0-notation-nonzero} and the definitions of \( G \) and \( C_6 \). The above display contradicts \eqref{eq:forcontradiction-1st-face}. Thus, Case (II) also cannot happen.

		Summarizing the above, we conclude that \eqref{eq:forcontradiction-1st-face} cannot happen. Thus, in view of {\cite[Lemma 3.12]{LiLoPo20}}, we must indeed have \( \gamma_{\bm{z},\eta} \in (0,\infty] \) and that the desired error bound follows from {\cite[Theorem 3.10]{LiLoPo20}}.
	\end{proof}

	\begin{remark}[Optimality of the error bound in Theorem~\ref{thm:Holderian-1/2-1st-type-face}]\label{remark:F-r-order-1/2}
		\!\!\!Let \( \bm{z}\!\in\! \!\bd\dgK\backslash\!\{\bm{0}\} \) with \( \bbz \neq \bm{0} \) and let \( \Fr := \{\bm{z}\}^\perp\cap \gK \). Then necessarily \( \tz_i > 0 \) for all \( i \). Moreover, we also know from the definition that \( \alpha_i > 0 \) for all \( i \). Now, consider the continuous function \( \bm{q}:(0,\alpha_1)\to \{\bm{z}\}^\perp \) defined by $\epsilon \mapsto \bm{q}_\epsilon:=(\bar{\bm{q}}_\epsilon,\widetilde{\bm{q}}_\epsilon)$ where
		\[
		\bar{\bm{q}}_\epsilon = -\bbz/\|\bbz\|^2,\,\,(\widetilde{\bm{q}}_\epsilon)_1 = (\alpha_1 - \epsilon)\tz_1^{-1},\,\,(\widetilde{\bm{q}}_\epsilon)_2 = (\alpha_2 + \epsilon)\tz_2^{-1},\,\mbox{and} \,(\widetilde{\bm{q}}_\epsilon)_i = \alpha_i \tz_i^{-1}, \,\forall i \ge 3.
		\]
		Notice that \( \bm{q}_\epsilon \) only differs from the \( \bm{f} \) in \eqref{eq:1st-type-face-ray} in two entries. One can check that \( \inProd{ \bm{z}}{\bm{q}_\epsilon} = 0 \) and \( \bm{q}_\epsilon\to \bm{f}\in \Fr\backslash\{\bm{0}\} \) as $\epsilon \downarrow 0$.
		Moreover, we have
		\[
		\begin{aligned}
			&\prod_{i=1}^n (\widetilde{\bm{q}}_\epsilon)_i^{\alpha_i}  = (\alpha_1 - \epsilon)^{\alpha_1}(\alpha_2 + \epsilon)^{\alpha_2}\tz_1^{-\alpha_1}\tz_2^{-\alpha_2}\prod_{i=3}^n\left(\frac{\alpha_i}{\tz_i}\right)^{\alpha_i}          \\
			& =  \left(1 - \frac{\epsilon}{\alpha_1}\right)^{\alpha_1}\left(1 + \frac{\epsilon}{\alpha_2}\right)^{\alpha_2}\prod_{i=1}^n\left(\frac{\alpha_i}{\tz_i}\right)^{\alpha_i}
			\overset{\rm (a)}= \left(1 - \frac{\epsilon}{\alpha_1}\right)^{\alpha_1}\left(1 + \frac{\epsilon}{\alpha_2}\right)^{\alpha_2} \|\bbz\|^{-1}                                                                                          \\
			& = \left(1 - \frac{\epsilon}{\alpha_1}\right)^{\alpha_1}\left(1 + \frac{\epsilon}{\alpha_2}\right)^{\alpha_2}\|\bar{\bm{q}}_\epsilon\|,
		\end{aligned}
		\]
		where (a) holds because \( \bm{z}\in \bd\dgK\backslash\{\bm{0}\} \) with \( \bbz\neq \bm{0} \). In view of this, if we define a continuous function \( \bm{p}:(0,\alpha_1)\to \gK \)  by $\epsilon \mapsto \bm{p}_\epsilon := (\bar{\bm{p}}_\epsilon,\widetilde{\bm{p}}_\epsilon)$ where
		\[
		\bar{\bm{p}}_\epsilon := -\left(1 - \frac{\epsilon}{\alpha_1}\right)^{\alpha_1}\left(1 + \frac{\epsilon}{\alpha_2}\right)^{\alpha_2}\frac{\bbz}{\|\bbz\|^2}\quad{\rm and}\quad \widetilde{\bm{p}}_\epsilon := \widetilde{\bm{q}}_\epsilon;
		\]
		then it is clear that \( \bm{p}_\epsilon \in \gK \), and we can compute that
		\begin{equation}\label{eq:distpower}
			\begin{aligned}
				\dist(\bm{q}_\epsilon,\gK) & \le \|\bm{q}_\epsilon - \bm{p}_\epsilon\| = \frac1{\|\bbz\|}\left|\left(1 - \frac{\epsilon}{\alpha_1}\right)^{\alpha_1}\left(1 + \frac{\epsilon}{\alpha_2}\right)^{\alpha_2} - 1\right| \\
				& = \frac1{\|\bbz\|}\left|(1 - \epsilon + O(\epsilon^2))(1 + \epsilon + O(\epsilon^2)) - 1\right| = O(\epsilon^2).
			\end{aligned}
		\end{equation}
		Next, we estimate \( \dist(\bm{q}_\epsilon,\Fr) \).
		Notice that \( \inProd{ \bm{q}_\epsilon}{\bm{f}} > 0 \) for all sufficiently small \( \epsilon \) because \( \bm{q}_\epsilon \to \bm{f} \). Hence, using the definition of \( \Fr \) and {Lemma \ref{lemma:Lemma-2.11}}, we see that
		\[
		\dist(\bm{q}_\epsilon,\Fr)^2 = \left\|\bm{q}_\epsilon - \frac{\inProd{ \bm{q}_\epsilon}{ \bm{f}}}{\|\bm{f}\|^2}\bm{f}\right\|^2 = \|\bm{q}_\epsilon\|^2 - \frac{(\inProd{ \bm{q}_\epsilon}{ \bm{f}})^2}{\|\bm{f}\|^2}.
		\]
		A direct computation then shows that
		\begin{align}
			\|\bm{q}_\epsilon\|^2 & = \frac1{\|\bbz\|^2} + (\alpha_1-\epsilon)^2\tz_1^{-2} + (\alpha_2+\epsilon)^2\tz_2^{-2} + \sum_{i=3}^n \alpha^2_i \tz_i^{-2}                       \notag \\
			& = \frac1{\|\bbz\|^2} + \sum_{i=1}^n \alpha^2_i \tz_i^{-2} + 2\epsilon(\alpha_2\tz_2^{-2} - \alpha_1\tz_1^{-2}) + \epsilon^2(\tz_1^{-2} + \tz_2^{-2}) \notag\\
			& = \|\bm{f}\|^2 + 2\epsilon(\alpha_2\tz_2^{-2} - \alpha_1\tz_1^{-2}) + \epsilon^2(\tz_1^{-2} + \tz_2^{-2}),\notag
		\end{align}
		where the last equality follows from the definition of \( \bm{f} \) in \eqref{eq:1st-type-face-ray}.
		Furthermore,
		\begin{align}
(\inProd{ \bm{q}_\epsilon}{ \bm{f}})^2 & =\textstyle \left(\frac1{\|\bbz\|^2} + \alpha_1(\alpha_1-\epsilon)\tz_1^{-2} + \alpha_2(\alpha_2+\epsilon)\tz_2^{-2} + \sum_{i=3}^n \alpha^2_i \tz_i^{-2}\right)^2 \notag\\
			& = \left[ \|\bm{f}\|^2  + \epsilon (\alpha_2\tz_2^{-2} - \alpha_1\tz_1^{-2})\right]^2                                                                    \notag\\
			& = \|\bm{f}\|^4 + 2\epsilon\|\bm{f}\|^2 (\alpha_2\tz_2^{-2} - \alpha_1\tz_1^{-2}) + \epsilon^2(\alpha_2\tz_2^{-2} - \alpha_1\tz_1^{-2})^2.\notag
		\end{align}
		Combining the above three identities, we deduce further that
		\begin{align}\label{eq:lowerlowerboundpower}
			\dist(\bm{q}_\epsilon,\Fr)^2 & = \epsilon^2\left(\tz_1^{-2} + \tz_2^{-2} - \frac{(\alpha_2\tz_2^{-2} - \alpha_1\tz_1^{-2})^2}{\|\bm{f}\|^2}\right)\nonumber                          \\
			& \ge \epsilon^2\left(\tz_1^{-2} + \tz_2^{-2} - \frac{(\alpha_2\tz_2^{-2} - \alpha_1\tz_1^{-2})^2}{\alpha^2_2\tz_2^{-2} + \alpha^2_1\tz_1^{-2}}\right),
		\end{align}
		where the inequality follows from the definition of \( \bm{f} \). Now, notice that in \eqref{eq:lowerlowerboundpower}, the scalar term is strictly greater than zero, because
		\[
		(\alpha_2\tz_2^{-2} - \alpha_1\tz_1^{-2})^2 < (\alpha_2\tz_2^{-2} + \alpha_1\tz_1^{-2})^2 \le (\tz_1^{-2} + \tz_2^{-2})(\alpha^2_2\tz_2^{-2} + \alpha^2_1\tz_1^{-2}),
		\]
		where the strict inequality holds because \( \alpha_i\tz_i^{-2} > 0 \) for \( i = 1 \), \( 2 \), and the last inequality follows from the Cauchy-Schwarz inequality.
		This together with \eqref{eq:lowerlowerboundpower} shows that \( \dist(\bm{q}_\epsilon,\Fr) = \Omega(\epsilon) \). Combining this with \eqref{eq:distpower}, we obtain
		$
		\limsup_{\epsilon\downarrow 0}\frac{\dist(\bm{q}_\epsilon,\gK)^\frac12}{\dist(\bm{q}_\epsilon,\Fr)} < \infty
		$.
		Thus \( | \cdot | ^{\frac{1}{2}} \) satisfies the asymptotic optimality criterion {(cf. \cite[Definition 3.1]{LiLoPo21}}) for \( \gK \) and \( \bm{z} \){, which implies that the error bound is optimal in the sense of \cite[Theorem 3.2(b)]{LiLoPo21}.}
	\end{remark}


	We now look at the faces that are exposed by \( \bm{z}\in \bd\dgK\backslash\{\bm{0}\} \) with \( \bbz = \bm 0 \).
	\begin{theorem}\label{thm:xface}
		Let \( \bm{z}\in \bd\dgK\backslash\{\bm{0}\} \) with \( \bbz = \bm{0} \) and let \( \Fhatz := \{\bm{z}\}^\perp\cap \gK \). Let \( \stdIndex := \{i\,| \,\tz_i > 0\} \),\footnote{Since \( \bbz = \bm{0} \) and \( \bm{z}\in \bd\dgK\backslash\{\bm{0}\} \), we must have \( \emptyset\neq \stdIndex \subsetneq \{1,2,\ldots,n\} \).} \( \eta > 0 \) and define \( \beta:= \sum_{i\in \stdIndex}\alpha_i \) and
		\begin{equation}\label{eq:gamma-beta-2nd-type-face}
		 \!\gamma_{\bm{z},\eta} \!:=\! \inf_{\bm{v}}\left\{\frac{\|\bm{v} - \bm{w}\|^\beta}{\|\bm{u} - \bm{w}\|}\,\bigg| \,
\begin{array}{c}
\bm{v}\in \bd \gK\cap \B(\eta)\backslash \Fhatz,\,\bm{w} = P_{\{\bm{z}\}^\perp}(\bm{v}),\\
\bm{u} = P_{\Fhatz}(\bm{w}),\,\bm{u} \neq \bm{w}
\end{array}\right\}.
		\end{equation}
		Then it holds that \( \gamma_{\bm{z},\eta}\in (0,\infty] \) and that
		\[
		\dist(\bm{q},\Fhatz)\le \max\{2\eta^{1-\beta},2\gamma_{\bm{z},\eta}^{-1}\}\cdot\dist(\bm{q},\gK)^\beta\,\,\text{ whenever }  \bm{q}\in \{\bm{z}\}^\perp\cap \B(\eta) .
		\]
	\end{theorem}
	\begin{proof}

		{In view of \cite[Theorem 3.10]{LiLoPo20}}, we need only show that $\gamma_{\bm{z},\eta} > 0$.
		To that end, let \( \bm{v}\in \bd \gK\cap \B(\eta)\backslash \Fhatz \), \( \bm{w} = P_{\{\bm{z}\}^\perp}(\bm{v}) \), \( \bm{u} = P_{\Fhatz}(\bm{w}) \), and \( \bm{u} \neq \bm{w} \). Then a direct computation shows
		that
		\begin{equation}\label{eq:dist-w-v}
			\|\bm{w}- \bm{v}\| = \frac1{\|\bm{z}\|}|\inProd{ \bm{z}}{\bm{v}}| \overset{\rm (a)}= \frac1{\|\bm{z}\|}\sum_{i\in \stdIndex}\tz_i\tv_i \overset{\rm (b)}\ge \frac{\min_{i\in \stdIndex}\tz_i}{\|\bm{z}\|}\sum_{i\in \stdIndex}\tv_i
			\overset{\rm (c)}\ge \frac{\min_{i\in \stdIndex}\tz_i}{\|\bm{z}\|}\|\tbv_{\stdIndex}\|,
		\end{equation}
		where (a), (b) and (c) hold because \( \tv_i\ge 0 \) and \( \tz_i > 0 \) for all \( i\in \stdIndex \), with \( \|\tbv_{\stdIndex}\|:= \sqrt{\sum_{i\in \stdIndex}\tv^2_i} \) (note that \( \stdIndex \neq \emptyset \), thanks to \( \bbz = 0 \) and \( \bm{z}\neq \bm{0} \)). Next, notice that \( \bm{w} = \bm{v} - \frac{\inProd{ \bm{z}}{\bm{v}}}{\|\bm{z}\|^2}\bm{z} \). Using this and the definitions of \( \bm{z} \) and \( \stdIndex \), we deduce that
		\begin{equation}\label{eq:form-w-by-v}
			\bbw = \bbv,\,\,\,\tw_i = \tv_i - \frac{\tz_i}{\|\bm{z}\|^2}\bigg(\sum_{j\in \stdIndex}\tz_j\tv_j\bigg)\,\,\,\, \forall i\in \stdIndex\,\,\,\,\,\,\,\mbox{and\,\,\,\,\,\,\,} \tw_i = \tv_i \ge 0\,\,\,\,\forall i \notin \stdIndex.
		\end{equation}
		In view of this and the definition of \( \Fhatz \) in \eqref{eq:2nd-type-face-d-dim}, we see that \( \tu_i = \tw_i \) whenever \( i\notin \stdIndex \), and hence
		\begin{equation}\label{eq:dist-w-u}
			\|\bm{w} - \bm{u}\| = \sqrt{\|\bbw\|^2 + \sum_{i\in \stdIndex}\tw^2_i} \le \sqrt{\|\bbv\|^2 + n(1+\sqrt{n})^2\|\tbv_{\stdIndex}\|^2},
		\end{equation}
		where the inequality follows from \eqref{eq:form-w-by-v} and the fact that for each \( i\in \stdIndex \),
		\[
		\begin{aligned}
			|\tw_i| & = \bigg|\tv_i - \frac{\tz_i}{\|\bm{z}\|^2}\bigg(\sum_{j\in \stdIndex}\tz_j\tv_j\bigg)\bigg|\le \bigg(1 + \frac{|\tz_i|}{\|\bm{z}\|^2}\sum_{j\in \stdIndex}|\tz_j|\bigg)\|\tbv_{\stdIndex}\| \\
			& \le \left(1 + \frac{\sqrt{n}|\tz_i|}{\|\bm{z}\|}\right)\|\tbv_{\stdIndex}\| \le (1+\sqrt{n})\|\tbv_{\stdIndex}\|.
		\end{aligned}
		\]
		Next, note that we have
		\begin{equation}\label{eq:upbd-norm-bar-v}
			\|\bbv\|=%
			\prod_{i=1}^n(\tv_i)^{\alpha_i} %
			= \prod_{i\notin \stdIndex}\tv_i^{\alpha_i} \cdot \prod_{i\in \stdIndex}\tv_i^{\alpha_i}%
			\le \prod_{i\notin \stdIndex}\eta^{\alpha_i} \cdot\prod_{i\in \stdIndex}\|\tbv_{\stdIndex}\|^{\alpha_i}%
			= \eta^{1-\beta} \|\tbv_{\stdIndex}\|^\beta,
		\end{equation}
		where the inequality holds because \( \bm{v}\in \B(\eta) \). Combining \eqref{eq:dist-w-v}, \eqref{eq:dist-w-u} and \eqref{eq:upbd-norm-bar-v}, we deduce
		\[
		\begin{aligned}
			\|\bm{w} - \bm{u}\| & \le \sqrt{\|\bbv\|^2 + n(1+\sqrt{n})^2\|\tbv_{\stdIndex}\|^2} \le \|\bbv\| + (n+\sqrt{n})\|\tbv_{\stdIndex}\|                                                                                                     \\
			& \le \eta^{1-\beta} \|\tbv_{\stdIndex}\|^\beta + (n+\sqrt{n})\|\tbv_{\stdIndex}\| = (\eta^{1-\beta}  + (n+\sqrt{n})\|\tbv_{\stdIndex}\|^{1-\beta})\|\tbv_{\stdIndex}\|^\beta                                   \\
			& \overset{\rm (a)}\le \eta^{1-\beta}(n+1 + \sqrt{n})\|\tbv_{\stdIndex}\|^\beta \overset{\rm (b)}{\le}\frac{\eta^{1-\beta}(n+1 + \sqrt{n})\|\bm{z}\|^\beta}{\left(\min_{i\in_{\stdIndex}}\tz_i\right)^\beta}\|\bm{w}-\bm{v}\|^\beta.
		\end{aligned}
		\]
		Here (a) holds since \( \bm{v}\in \B(\eta) \) and \( \beta \in (0,1) \); (b) is true because of \eqref{eq:dist-w-v}.
		Thus, \( \gamma_{\bm{z},\eta}\ge \frac{\left(\min_{i\in \stdIndex}\tz_i\right)^\beta}{\eta^{1-\beta}(n+1 + \sqrt{n})\|\bm{z}\|^\beta} > 0 \), and the desired error bound {follows from \cite[Theorem 3.10]{LiLoPo20}}.
	\end{proof}

	\begin{remark}[Optimality of the error bound in Theorem~\ref{thm:xface}]\label{remark:F-z-order-beta}
		\!\!\!Let \( \bm{z}\!\in\!\! \bd\dgK\backslash\!\{\bm{0}\} \) with \( \bbz = \bm{0} \) and let \( \Fhatz := \{\bm{z}\}^\perp\cap \gK \). Let \( \stdIndex := \{i\,| \,\tz_i > 0\}\neq \emptyset \) and define
		\[
		\beta:= \sum_{i\in \stdIndex}\alpha_i\in (0,1).
		\]
		Fix any \( \bm{u}\in \R^m \) with \( \|\bm{u}\|=1 \) and define the continuous function \( \bm{q}:(0,1)\to \{\bm{z}\}^\perp \) by $\epsilon \mapsto \bm{q}_\epsilon := (\bar{\bm{q}}_\epsilon,\widetilde{\bm{q}}_\epsilon)$ where
		\[
		\bar{\bm{q}}_\epsilon = \epsilon^\beta \bm{u},\,\,(\widetilde{\bm{q}}_\epsilon)_i = 0\,\,\forall i \in \stdIndex,\,\mbox{and} \,(\widetilde{\bm{q}}_\epsilon)_i = 1, \,\forall i \notin \stdIndex.
		\]
		It is clear that for all $\epsilon$, \( \inProd{ \bm{z}}{\bm{q}_\epsilon} = 0 \) and \( \dist(\bm{q}_\epsilon,\Fhatz) \to 0 \) as $\epsilon \downarrow 0$ . Now, define the function \( \bm{p}:(0,1)\to \gK \) by $\epsilon \mapsto \bm{p}_\epsilon := (\bar{\bm{p}}_\epsilon,\widetilde{\bm{p}}_\epsilon)$ where
		\[
		\bar{\bm{p}}_\epsilon = \epsilon^\beta \bm{u},\,\,(\widetilde{\bm{p}}_\epsilon)_i = \epsilon\,\,\forall i \in \stdIndex,\,\mbox{and} \,(\widetilde{\bm{p}}_\epsilon)_i = 1, \,\forall i \notin \stdIndex.
		\]
		Clearly $\bm{p}_\epsilon$ lies in \( \gK \), and we have that
		$
			\dist(\bm{q}_\epsilon,\gK)\le \|\bm{q}_\epsilon - \bm{p}_\epsilon\| \le |\stdIndex|\cdot\epsilon
		$.
		On the other hand, we have in view of \eqref{eq:2nd-type-face-d-dim} that \( \dist(\bm{q}_\epsilon,\Fhatz) = \epsilon^\beta > 0 \).
		Hence,
		$
		\limsup_{\epsilon\downarrow 0}\frac{\dist(\bm{q}_\epsilon,\gK)^\beta}{\dist(\bm{q}_\epsilon,\Fhatz)} \le |\stdIndex|^\beta < \infty
		$.
		Thus \( | \cdot | ^{\beta} \) satisfies the asymptotic optimality criterion {(cf. \cite[Definition 3.1]{LiLoPo21}}) for \( \gK \) and \( \bm{z} \){, which implies that the error bound is optimal in the sense of \cite[Theorem 3.2(b)]{LiLoPo21}.}
	\end{remark}

	Using Theorems~\ref{thm:Holderian-1/2-1st-type-face} and \ref{thm:xface} {together with \cite[Lemma~3.9]{LiLoPo20}}, we have the following result concerning one-step facial residual functions.
	\begin{corollary}\label{corollary:1-FRFs-power-cones}
		Consider \( \gK \) and its dual cone \( \dgK \).
		\begin{enumerate}[label=(\roman*)]
			\item\label{3.6i} Let \( \bm{z}\in \bd\dgK\backslash\{\bm{0}\} \) with \( \bbz\neq \bm{0} \) and let \( \Fr := \{\bm{z}\}^\perp\cap \gK \). Let \( \gamma_{\bm{z},t} \) be defined as in \eqref{eq:gamma-1/2-1st-type-face}. Then the function \( \psi_{\gK,\bm{z}}:\R_+\times \R_+\to \R_+ \) given by\vspace{-0.15 cm}
			\begin{equation}\label{eq:1-FRF-1st-face}\vspace{-0.15 cm}
				\!\!\psi_{\gK,\bm{z}}(\epsilon,t) \!:=\! \max \left\{\epsilon,\epsilon/\|\bm{z}\| \right\} + \max\{2\sqrt{t},2\gamma_{\bm{z},t}^{-1}\}(\epsilon +\max \left\{\epsilon,\epsilon/\|\bm{z}\| \right\} )^\frac12
			\end{equation}
			is a one-step facial residual function for \( \gK \) and \( \bm{z} \).
			\item Let \( \bm{z}\in \bd\dgK\backslash\{\bm{0}\} \) with \( \bbz = \bm{0} \) and let \( \Fhatz := \{\bm{z}\}^\perp\cap \gK \). Let \( \gamma_{\bm{z},t} \) be defined as in \eqref{eq:gamma-beta-2nd-type-face}, where \( \beta:= \sum_{i:\tz_i > 0}\alpha_i \). Then the function \( \psi_{\gK,\bm{z}}:\R_+\times \R_+\to \R_+ \) given by\vspace{-0.15 cm}
			\begin{equation}\label{eq:1-FRF-2nd-face}\vspace{-0.15 cm}
				\!\!\psi_{\gK,\bm{z}}(\epsilon,t) \!:= \!\max\! \left\{\epsilon,\epsilon/\|\bm{z}\| \right\} + \max\{2t^{1-\beta}\!,2\gamma_{\bm{z},t}^{-1}\}(\epsilon +\max \left\{\epsilon,\epsilon/\|\bm{z}\| \right\} )^\beta
			\end{equation}
			is a one-step facial residual function for \( \gK \) and \( \bm{z} \).
		\end{enumerate}
	\end{corollary}

	We now collect these results to show the tight error bounds for $\gK$.
	\begin{theorem}[Error bounds for the generalized power cone and their optimality]\label{thm:error-bound-power-cone-optimality}
		Consider \( \gK \) and its dual cone \( \dgK \).
		Let \( \subSpace \subseteq \R^{m + n} \) be a subspace and \( \bm{a} \in \R^{m+n} \) be given.
		Suppose that \( ( \subSpace + \bm{a} ) \cap \gK \neq \emptyset \).
		Then the following items hold.
		\begin{enumerate}[label=(\roman*)]
			\item\label{item:error-bound-power-cone-1} \( \dpps (\gK, \subSpace + \bm{a}) \leq 1 \).
			\item\label{item:error-bound-power-cone-2} If \( \dpps (\gK, \subSpace + \bm{a}) = 0 \), then a Lipschitzian error bound holds.
			\item\label{item:error-bound-power-cone-3} If \( \dpps (\gK, \subSpace + \bm{a}) = 1 \), consider the chain of faces \( \stdFace \subsetneq \gK \) with length being \( 2 \).
			\begin{enumerate}[label=(\alph*)]
				\item\label{item:error-bound-power-cone-3-1} If \( \stdFace = \Fr \), then a H\"olderian error bound with exponent \( 1 / 2 \) holds.
				\item\label{item:error-bound-power-cone-3-2} If \( \stdFace = \Fhatz \) with \( \bm{z} \in \dgK \cap \subSpace^{\perp } \cap \{\bm{a}\}^{\perp } \), then a H\"olderian error bound with exponent \( \beta := \sum_{i : \tz_i > 0} \alpha _i \) holds.
				\item\label{item:error-bound-power-cone-3-3} If \( \stdFace = \{\bm{0}\} \), then a Lipschitzian error bound holds.
			\end{enumerate}
			\item\label{item:error-bound-power-cone-4} All these error bounds are the best in the sense {stated in \cite[Theorem 3.2(b)]{LiLoPo21}}.
		\end{enumerate}
	\end{theorem}

	\begin{proof}
		As is shown in Section \ref{subsec:facial-structure}, all the proper exposed faces of the generalized power cone are polyhedral.
		Then the process of facial reduction needs at most one step to reach the PPS condition. Hence, \( \dpps(\gK, \subSpace + \bm{a}) \leq 1 \).
		This shows item \ref{item:error-bound-power-cone-1}.

		If \( \dpps(\gK, \subSpace + \bm{a}) = 0 \), i.e., \eqref{eq:conic-feasibility-problem} satisfies the PPS condition, then by \cite[Corollary 3]{BBL99}, a Lipschitzian error bound holds.
		This shows item \ref{item:error-bound-power-cone-2}.

		Next, let \( \dpps(\gK, \subSpace + \bm{a}) = 1 \); i.e., we need one step to reach the PPS condition.
		In this case, the error bound depends on the exposed face \( \stdFace \) that contains the feasible region.
		If \(\stdFace =  \Fr \), then by Corollary \ref{corollary:1-FRFs-power-cones}\ref{3.6i}, we conclude that a H\"olderian error bound with exponent \( 1 / 2 \) holds.
		Remark \ref{remark:F-r-order-1/2} implies that \( \mathfrak{g} = | \cdot | ^{\frac{1}{2}}\) satisfies the asymptotic optimality criterion for \( \gK \) and \( \bm{z} \) with \( \bbz \neq \bm{0} \).
		Hence, by {\cite[Theorem 3.2]{LiLoPo21}}, the obtained H\"olderian error bound with exponent \( 1 / 2 \) is the best error bound.

		If \( \stdFace = \Fhatz \) with \( \bm{z} \in \dgK \cap \subSpace^{\perp } \cap \{\bm{a}\}^{\perp } \), then using Corollary \ref{corollary:1-FRFs-power-cones}(ii), we conclude that a H\"olderian error bound with exponent \( \beta := \sum _{i \in \stdIndex} \alpha _i \) holds, where \( \stdIndex = \{i\,| \,\tz_i > 0\} \).
		The optimality of this error bound comes from Remark \ref{remark:F-z-order-beta} and {\cite[Theorem 3.2]{LiLoPo21}}.
If \( \stdFace = \{\bm{0}\} \), which means the feasible region is \( \{\bm{0}\} \), then a Lipschitzian error bound holds automatically and it is naturally tight, see \cite[Proposition~27]{L17}.
	\end{proof}

	\section[Application: Self-duality, homogeneity, irreducibility and perfectness of the generalized power cone]{Application: Self-duality, homogeneity, irreducibility and perfectness of \( \gK \)}\label{sec:applications}
	In this section, we consider the \textit{self-duality}, \textit{homogeneity}, \textit{irreducibility} and \textit{perfectness} of \( \gK \).
	We first briefly explain the importance of those questions.

	In what follows, we need the following concepts.
	We will denote by $\aut(\stdCone)$ the group of automorphisms of $\stdCone$ which are the linear bijections $\bm{M}:\EucSpace\to \EucSpace$ such that $\bm{M} \stdCone = \stdCone$. Then, the Lie algebra of
	$\aut(\stdCone)$ denoted by $\lie \aut(\stdCone)$ corresponds to the linear maps $\bm{L}$ for which $e^{t\bm{L}} \in \aut(\stdCone)$ for all $t \in \R$ or, equivalently, is the tangent space at the identity element when $\aut (\stdCone)$ is seen as a Lie group.
	
		Recall that a cone \( \stdCone \) is called \textit{self-dual} if there exists a positive definite matrix \( \bm{Q} \) such that \( \bm{Q}\stdCone = \stdCone^{*} \).
	This is equivalent to the existence of \emph{some} inner product under which $\stdCone$ becomes self-dual, e.g., \cite[Proposition~1]{IL19}.	
	A cone is \textit{homogeneous} if for every \( \bm{x}, \bm{y} \in \ri \stdCone \), there is a matrix \( \bm{A} \in \aut(\stdCone) \) such that \( \bm{A}\bm{x} = \bm{y} \).
	A homogeneous and self-dual cone is called \textit{symmetric} \cite{FK94}.

	If a closed convex cone \( \stdCone \) can be expressed as a direct sum of two nonempty and nontrivial sets \( \stdCone_1, \stdCone_2 \subset \stdCone \), i.e., \( \stdCone = \stdCone_1 + \stdCone_2 \) with \( \stdCone_1 \neq \{\bm{0}\}, \stdCone_2 \neq \{\bm{0}\} \) and \( \Span(\stdCone_1) \cap \Span(\stdCone_2) = \{\bm{0}\} \), then \( \stdCone \) is said to be \textit{reducible}; it might not be immediately obvious, but this forces $\stdCone_1$ and $\stdCone_2$ to be convex cones, e.g., \cite[Lemma~3.2]{LS75}.
	Otherwise, \( \stdCone \) is said to be \textit{irreducible} or \textit{indecomposable}, e.g., \cite{LS75,Ba81,GT14}.

	\subsection{Some theoretical context}\label{sec:context}
		It is relatively recent that the power cone has been a subject of research in optimization. However, the power cone was first considered in the 50's by Max Koecher in the context of the so-called \emph{domains of positivity}, see \cite{Ko57}. 
	More precisely, Koecher  proposed a family of 3D cones in \cite[Section~11,d)]{Ko57} which corresponds to $\gKmn{1}{2}$, with $\alpha \in (0,1)$.
	After that, the power cone languished in relative obscurity inside the optimization community, although it was discussed briefly in \cite{TX01} and in \cite{TT03} under the name of \emph{Koecher cone}.	
	As indicated in the introduction, several works helped to revitalize the interest in power cones by showcasing modelling applications, algorithms and software \cite{Ch09, MC2020, SY15,KT19,PY21,CKV21}.

	When the power cone is bundled together in the class
	of ``non-symmetric cones'', it might be interesting to take a step back and understand two points: (a) how exactly the power cone fails to be symmetric and (b) why one should care about  this.
	
	Starting from the latter, what is special about symmetric cones is that they are supported by a powerful theory of \emph{Jordan algebras} \cite{FK94}.	Being a symmetric cone is a \emph{very} {favourable} property which was heavily exploited to develop efficient primal-dual interior point algorithms, e.g., \cite{FB08}. However, being a symmetric cone is also restrictive for {it is} known that, up to linear isomorphism, each symmetric cone is a direct product of only five types of cones. The most remarkable examples of symmetric cones are the $\R^n_+$, the real symmetric positive semidefinite matrices $\mathcal{S}^n_+$, the second-order cone and the direct products of those three.
	
	As for item (a), examining \eqref{eq:gpowercones}, we immediately see that   the dual of $\gK$ under the Euclidean inner product is just $\bm{D}\gK$, where $\bm{D}$ is a diagonal matrix with positive entries, so $\gK$ is indeed self-dual in the sense above. Thus the only gap between $\gK$ and the class of symmetric cones is the homogeneity.
	
	Given that being symmetric is very advantageous, one may reasonably wonder if the family of cones $\gK$ parametrized by $\bm{\alpha}$ and $m$ and $n$ are indeed non-homogeneous in general.
	To the best of our knowledge, although it is well-known (e.g., see comments in \cite[Section~4]{TX01}) that $\gKmn{1}{2}$ is non-homogeneous except when ${\bm\alpha} = (1/2,1/2)$,  there is no result on the generalized power cone regarding which combination of the parameters $m$, $n$ and $\bm\alpha$ leads to homogeneity or not. We fill this gap with Theorem~\ref{thm:dimension-Aut(P)} and Corollary~\ref{corollary:irreducibility-nonhomogeneous-perfect}, which tells us precisely which of the generalized power cones are homogeneous or not.
	
	We also completely determine the automorphism group of
	$\gK$. While this may seem an esoteric question, the automorphism group of a cone $\stdCone$ is intimately connected to complementarity questions over $
	\stdCone$. For example, it is known that $\bm{L}$ belongs to the Lie algebra of $\aut(\stdCone)$ if and only if the following implication holds
	\[
	\bm{x}\in \stdCone, \bm{y} \in \stdCone^*,	\inProd{\bm{x}}{\bm{y}} = 0 \Rightarrow \inProd{\bm{L}\bm{x}}{\bm{y}} = 0,
	\]
	see \cite{GT2014}. If a cone has ``enough'' automorphisms then a complementarity problem can be rewritten as a square system using the matrices from the Lie algebra of $\aut(\stdCone)$.
	In particular, when the dimension of $\aut(\stdCone)$ is at least $\dim \stdCone$, then the cone is said to be \emph{perfect}, see \cite[Page 5]{GT2014} and \cite[Theorem~1]{OG2016}. An example of this phenomenon is how the conditions $\bm{x},\bm{y} \in \R^n_+, \inProd{\bm{x}}{\bm{y}} = 0$ imply $n$ equations ${x}_i {y}_i = 0$ which is useful in several contexts.
	
	The quantity $\dim \aut \stdCone$ is called the \emph{Lyapunov rank of $\stdCone$} \cite{GT2014,GT14} and is additive with respect to direct sums \cite[Proposition~1]{GT2014}.
	Since any cone can be written as a direct sum of irreducible cones, it becomes important to identify \emph{which} irreducible cones are perfect.
	
	It is interesting to note that many of the examples of irreducible perfect cones in the literature (e.g., \cite{GT2014,GT14,OG2016}) seem to be homogeneous. In addition, every homogeneous cone is perfect, which follows by known results about Lie groups, e.g., see \cite[Theorem~21.20]{L2012} or Section~2 in \cite{Or22} which summarizes useful results. The final observation we will make in this paper is that, surprisingly, for some choices of parameters, $\gK$ is perfect but non-homogeneous, see Corollary~\ref{corollary:irreducibility-nonhomogeneous-perfect}. We note that in \cite{Sz16}, Sznajder showed that there are choices of parameters for which the so-called \emph{extended second order cone} is irreducible and perfect. This corresponds to a family of cones proposed by N\'emeth and Zhang that contains the second order cones \cite{NZ15}. However, as far as we know, the homogeneity of those cones (or the lack thereof) was not discussed in general.

	\subsection{Automorphisms of the generalized power cone}%
In this subsection, we will prove our main results regarding
$\aut(\gK)$. The basic strategy is simple: if $\bm{A} \in \aut(\gK)$, then $\bm{A}$ must map a face $\stdFace_1$ of $\gK$ to another face $\stdFace_2$ of $\gK$ with the same properties such as the dimension.
More than that, the optimal exponents associated to FRFs of $\stdFace_1$ and $\stdFace_2$ must be the same. These conditions impose enough restrictions on $\bm{A}$ that we are able to completely determine its shape. Note that when $n=2$ and ${\bm\alpha}=(1/2,1/2)$, $\gK$ is isomorphic to the second-order cone, whose automorphism group is well-known. Below, we focus on the complementary cases.


	\begin{theorem}[Automorphisms of \( \gK \)]\label{thm:self-duality-homogeneity}
			For \( m \geq 1, n > 2 \) and any \( \bm{\alpha} \in (0, 1)^n \) such that \( \sum_{i=1}^n \alpha _i = 1  \), or for \( m \geq 1, n = 2 \) and any \( \bm{\alpha} \in (0, 1)^2 \) such that \( \alpha_1 \neq \alpha _2 \) and \( \alpha _1 + \alpha _2 = 1 \), it holds that $\bm{A} \in \aut(\gK)$ if and only if
			\begin{equation}
				\label{eq:form-A-AK=K}
				\bm{A} =
				\begin{bmatrix}
					\bm{B}      & \bm{0} \\
					\bm{0} & \bm{E}
				\end{bmatrix}
			\end{equation}
       for some (invertible) {generalized permutation matrix}\footnote{A generalized permutation matrix is a matrix where in each column and each row there is exactly one nonzero entry.} {$\bm{E} \in \R^{n \times n}$} with positive nonzero entries and invertible matrix {$\bm{B} \in \R^{m \times m}$} satisfying \( \| \bm{B} \bm{x} \| = \prod_{k=1}^n (E_{k, l_k})^{\alpha _{l_k}} \| \bm{x} \|\) for all \( \bm{x} \in \R^m \),
       where $E_{k,l_k}$ is the nonzero element in the $k$-th row of $\bm{E}$ and $\alpha_{l_k} = \alpha_k$.
	\end{theorem}

	\begin{proof}
		Suppose that there exists a matrix
		\[ \bm{A} :=
		\begin{bmatrix}
			\bm{B} & \bm{C} \\
			\bm{D} & \bm{E}
		\end{bmatrix} {\text{ with } \bm{B} \in \R^{m \times m}, \,\bm{C} \in \R^{m \times n}, \,\bm{D} \in \R^{n \times m}, \,\bm{E} \in \R^{n \times n}} \]
		such that \( \bm{A}\gK = \gK \).

		First note that the entries of $\bm{E}$ must all be nonnegative, for if the $(i,j)$-th entry was negative, then we could pick a vector \( \bm{q} := (\bm{0}, \bm{c}) \in \gK \) with $c_j=1$ and $c_k=0$ for $k \neq j$, wherefore \( \bm{A}\bm{q} \not\in \gK \), which is a contradiction.

		Additionally, such a matrix $\bm{A}$ must be invertible and if \( \psi  \) is an FRF for a face \( \stdFace_1 \unlhd \gK \), then \( \bm{A} \) must map \( \stdFace_1 \) onto a face \( \stdFace_2 \unlhd \gK \) which has the same dimension and admits an FRF that is a positively rescaled shift of \( \psi  \); see~\cite[Proposition 17]{L17}.

		Observe from Section~\ref{subsec:facial-structure} that the generalized power cone has two types of faces defined in \eqref{eq:1st-type-face-ray} and \eqref{eq:2nd-type-face-d-dim} (denoted by $\Fr$ and $\Fhatz$ respectively with an abuse of notation) with the corresponding (optimal) one-step facial residual functions in \eqref{eq:1-FRF-1st-face} and \eqref{eq:1-FRF-2nd-face}, respectively.
		We also notice that the dimension of the faces of the first type is 1, while the dimension of a face of the second type is \( n - | \stdIndex | \). These lead to the following observations:
\begin{enumerate}[label=(\Roman*)]
  \item Given an \( \stdIndex \) with \( \beta _{\stdIndex} := \sum_{i \in \stdIndex} \alpha _i  \), if \( | \stdIndex | < n - 1 \), i.e., the dimension of the corresponding face is larger than \( 1 \), then \( \bm{A} \) must map the face associated with \( \stdIndex \) to a face associated with an \( \bar{\stdIndex} \) where \( | \stdIndex |  = |\bar{\stdIndex}| \) and \( \beta _{\bar{\stdIndex}} = \beta _\stdIndex \).
  \item In the case when $n = 2$, since we assumed $\alpha_1 \neq \alpha_2$ and thus $\alpha_1\neq 1 / 2$, $\bm{A}$ cannot map a one-dimensional face of type $\Fhatz$ (whose FRF admits an optimal exponent of $\alpha_1$ or $\alpha_2$) to one of type $\Fr$ (whose FRF admits an optimal exponent of $1 / 2$).
\end{enumerate}

Thus, a face of type $\Fhatz$ with $|\stdIndex| = 1$ must be mapped to a face of the same type.
From now on, for each $k\in \{1,2,\ldots,n\}$, we let $i_k$ and $l_k$ be such that $\bm{A} \stdFace_{\{k\}} = \stdFace_{\{i_k\}}$ and $\bm{A}\stdFace_{\{l_k\}} =  \stdFace_{\{k\}}$, where $\stdFace_{\{k\}}$ denotes the face of type $\Fhatz$ associated with $\stdIndex = \{k\}$. We deduce immediately from the above discussions that $\{1,2,\ldots,n\} = \{i_1,i_2,\ldots,i_n\} = \{l_1,l_2,\ldots,l_n\}$ and $\alpha_k = \alpha_{i_k} = \alpha_{l_k}$.

		Now,  fix any \( k \in \{1, 2, \dots , n\} \).
		Then for any \( \tbx_{\stdIndex} := (c_1,\dots, c_{k-1}, 0, c_{k+1},\dots, c_n) \) with $c_i > 0$ for all $i \neq k$, it must hold that \( \bm{A} \) maps \( \bm{x}_{\stdIndex} := (\bm{0}, \tbx_{\stdIndex}) \) to some \( \bm{x}_{\hat{\stdIndex}} := (\bm{0}, \tbx_{\hat{\stdIndex}}) \) with \( \hat{\stdIndex} = \{i_k\},\,\alpha _k = \alpha _{i_k} \) and \( (\tbx_{\hat{\stdIndex}})_{i_k} = 0 \).
		Thus,
		\[
		\begin{bmatrix}
			\bm{B} & \bm{C} \\
			\bm{D} & \bm{E}
		\end{bmatrix}
		\begin{bmatrix}
			\bm{0} \\
			\tbx_{\stdIndex}
		\end{bmatrix} =
		\begin{bmatrix}
			\bm{0} \\
			\tbx_{\hat{\stdIndex}}
		\end{bmatrix}.
		\]
		Therefore, we have \( \bm{C}\tbx_{\stdIndex} = \bm{0} \). This together with the arbitrariness of $c_i > 0$ shows that all except possibly the \( k \)-th column of \( \bm{C} \) are \( \bm{0} \).
		Since \( k \) is arbitrary, then we conclude that \( \bm{C} = \bm{0} \).

		Next, notice that we also have \( \bm{E}\tbx_{\stdIndex} = \tbx_{\hat{\stdIndex}} \). Since \( (\tbx_{\hat{\stdIndex}})_{i_k} = 0 \), we see that \( E_{i_k} \tbx_{\stdIndex} = 0 \), where \( E_{i_k} \) is the \( i_k \)-th row of \( \bm{E} \). Using again the arbitrariness of $c_i > 0$ in the definition of $\tbx_{\stdIndex}$, we conclude that all entries of \( E_{i_k} \) are 0 except possibly for the \( k \)-th entry, i.e., \( E_{i_k} \) has only one possibly nonzero entry and that entry is nonnegative.
		From the arbitrariness of \( k \) and the fact that \( \{i_1, i_2, \dots ,i_n\} = \{1, 2, \dots , n\} \), we immediately obtain that every entry of the $i_k$-th row $\bm{E}$ has all of its entries equal to zero except possibly for the $k$-th, which is nonnegative.

		Taking into account of the fact that $\bm{A}$ is invertible and $\bm{C}=\bm{0}$, we know that none of the columns of $\bm{E}$ can be identically zero, and so we altogether have that each of the rows and columns of $\bm{E}$ consists of one strictly positive entry, with all other entries identically zero.\footnote{\label{footnote:k,ik,lk}Then, we have shown that \( E_{s, r} \neq 0 \) if and only if \( (s, r) = (i_k, k) \) for some \( k \in \{1, 2, \dots , n\} \) (or equivalently \( (s, r) = (k, l_k) \) for some \( k \in \{1, 2, \dots , n\} \)). }

We next claim that $\bm{A}$ must map faces of type $\Fr$ to a face of type $\Fr$. Since $\bm{A}$ must permute faces whose FRFs admit the same optimal exponent, we only need to consider the extreme case that there exists a face of type $\Fhatz$ corresponding to an \( \stdIndex := \{1, 2, \dots , i - 1, i + 1, \dots , n\} \) for some \( i \) (i.e., the dimension of the corresponding face is 1) with \( \beta _{\stdIndex} = 1 / 2 \), and argue that $\bm{A}$ cannot map $\Fr$ onto such $\Fhatz$.
Suppose for contradiction that this happens; then there must exist \( {\bm x} = (\bbx, \tbx) \) in some face of type $\Fr$ with \( \bbx \neq \bm{0} \) and \( \tx_i > 0 \) for all $i$ such that
		\[
		\begin{bmatrix}
			\bm{B} & \bm{0} \\
			\bm{D} & \bm{E}
		\end{bmatrix}
\begin{bmatrix}
			\bbx \\
			\tbx
		\end{bmatrix}
		=
\begin{bmatrix}
			\bm{0} \\
			\bm{e}_{i}
		\end{bmatrix},
		\]
		where \( \bm{e}_{i}\in \R^n \) is the vector whose elements are all zero except for the \( i \)-th element being 1.
		However, this cannot happen because $\bm{B}\bbx = {\bm 0}$ and the invertibility of $\bm{B}$ (a consequence of invertibility of $\bm{A}$) implies $\bbx = {\bm 0}$, leading to a contradiction.
		Hence, \( \bm{A} \) must map faces of type $\Fr$ onto a face of type $\Fr$.

		Thus, for any \( \bm{x} = (\bbx, \tbx) \) in one of the type $\Fr$ faces with \( \bbx \neq \bm{0}\), \(\min_i\{\tx_i\} > {0} \) and \( \| \bbx \| = \prod_{i=1}^n \tx_i^{\alpha _i} \), there must be \( \bm{y} = (\bby, \tby) \) in one of the type $\Fr$ faces with \( \bby \neq \bm{0}\), \(\min_i\{\ty_i\} > {0} \) and \( \| \bby \| = \prod_{i=1}^n \ty_i^{\alpha _i} \) such that
		\[
		\begin{bmatrix}
			\bm{B} & \bm{0} \\
			\bm{D} & \bm{E}
		\end{bmatrix}
		\begin{bmatrix}
			\bbx \\
			\tbx
		\end{bmatrix} =
		\begin{bmatrix}
			\bby \\
			\tby
		\end{bmatrix}.
		\]
		Recall that there is exactly one nonzero element in each row of \( \bm{E} \), and this element is positive. From the definition of $l_k$, this nonzero element is $E_{k,l_k}$; see footnote \ref{footnote:k,ik,lk}.

		Fix any $j$ and $k\in \{1,\dots,n\}$. Pick any \((\bbx,\tbx)\in \gK\) such that \( \bbx = \bm{e}_j \) and $\prod^n_{i=1}\tx_i^{\alpha_i}=1$. For any $t > 0$, one can check that \((t^{\alpha_{l_k}}\bbx,\tx_1,\cdots,t\tx_{l_k},\cdots,\tx_n)\in \gK\) belongs to a face of type $\Fr$. Thus, there exists $(\bby,\tby)$ such that
		\[
			t^{\alpha_{l_k}}\bm{B} \bm{e}_j = \bby\ \ \ {\rm and}\ \ \
			t^{\alpha_{l_k}}D_{k,j} + tE_{k, l_k}\tx_{l_k}  = \ty_k > 0.
		\]
The second relation implies that $D_{k,j} + t^{1-\alpha_{l_k}}E_{k, l_k}\tx_{l_k} > 0$. Letting $t\downarrow 0$, we conclude that $D_{k,j}\ge 0$. As the choices of $j$ and $k$ were arbitrary, we see that all entries of $\bm{D}$ are nonnegative. Considering \( \bbx = -\bm{e}_j \), a similar argument shows that all entries of $\bm{D}$ are nonpositive. Hence, \( \bm{D} = \bm{0} \). 

		Now, for any $\bbx \in \R^m$, pick any \((\bbx,\tbx)\in \bd\gK\). Then there exists \((\bby,\tby)\in \bd\gK\) so that\footnote{Such a ${\bm y}$ exists because $\bm{A}$ is invertible and $\bm{A}\gK = \gK$, which implies $\bm{A} \ri\gK = \ri\gK$ and $\bm{A}\bd \gK = \bd \gK$.}
		\[
			\bm{B} \bbx = \bby \ \ \ {\rm and}\ \ \
			E_{k, l_k}\tx_{l_k}  = \ty_k\ \ {\rm for}\ k = 1, 2, \dots, n.
		\]
		Thus,
		\[
\begin{aligned}
  \| \bm{B}\bbx \| &= \| \bby \| = \prod_{k=1}^n \ty_k^{\alpha_k} = \prod_{k=1}^n (E_{k, l_k} \tx_{l_k})^{\alpha_{k}} \overset{\rm (a)}= \prod_{k=1}^n (E_{k, l_k} \tx_{l_k})^{\alpha_{l_k}}  \\
  &= \prod_{k=1}^n E_{k, l_k}^{\alpha _{l_k}} \prod_{i=1}^n \tx_i^{\alpha _i} = \prod_{k=1}^n E_{k, l_k}^{\alpha _{l_k}} \| \bbx \|.
\end{aligned}
\]
where (a) holds as $\alpha_{k} = \alpha_{l_k}$ for all $k$.
		The above shows the necessity of the form in \eqref{eq:form-A-AK=K}.

		Conversely, if \( \bm{A} \) is a matrix of the form \eqref{eq:form-A-AK=K}, then \( \bm{A} \) must be invertible since \( \bm{B} \) and \( \bm{E} \) are invertible.
		For any \( \bm{x} = (\bbx, \tbx) \in \gK \), we have \( \bm{A} \bm{x} = (\bm{B}\bbx, \bm{E}\tbx) \).
		Hence,
		\[ \| \bm{B}\bbx \| = \prod_{k=1}^n E_{k, l_k}^{\alpha _{l_k}} \| \bbx \| \leq \prod_{k=1}^n E_{k, l_k}^{\alpha _{l_k}} \prod_{i=1}^n \tx_i^{\alpha _i} =  \prod_{k=1}^n \left( E_{k, l_k} \tx_{l_k} \right)^{\alpha _{l_k}}, \]
		where the last equality holds as $\{1,\ldots,n\}= \{l_1,\ldots,l_n\}$.
		This implies \( \bm{A}\gK \subseteq \gK \).

		We claim
		\begin{equation}\label{eq:claims-E-inv-B-inv}
			\text{(i) } \left(\bm{E}^{-1}\right)_{i, j} =
			\begin{cases}
				0,                  & E_{j, i} = 0,    \\
				\frac{1}{E_{j, i}}, & E_{j, i} \neq 0.
			\end{cases}
			\quad \text{(ii) } \| \bm{B}^{-1}\bm{x} \| = \prod_{k=1}^n E_{k, l_k}^{-\alpha _{l_k}}\| \bm{x} \| \ \ \ \forall \bm{x} \in \R^m.
		\end{equation}
		Granting these, we have that for any \( \bm{x} = (\bbx, \tbx) \in \gK \), \( \bm{A}^{-1}\bm{x} = (\bm{B}^{-1}\bbx, E^{-1} \tbx) \) satisfies\vspace{-0.1 cm}
\begin{align}\vspace{-0.1 cm}
\prod_{i=1}^n (\bm{E}^{-1} \tbx)_i^{\alpha_i} &= \prod_{i=1}^n \bigg( \sum_{j=1}^n (\bm{E}^{-1})_{i, j} \tx_j \bigg)^{\alpha _i} \overset{{\rm (a)}}{=} \prod_{k=1}^n \left((\bm{E}^{-1})_{l_k,k} \tx_k\right)^{\alpha_{l_k}} = \prod_{k=1}^n (E_{k,l_k}^{-1} \tx_k)^{\alpha_{l_k}} \notag\\
&= \prod_{k=1}^n  E_{k,l_k}^{-\alpha_{l_k}}  \prod_{i=1}^n  \tx_i ^{\alpha _{l_i}} \overset{\rm (b)}=  \prod_{k=1}^n  E_{k,l_k}^{-\alpha_{l_k}}  \prod_{i=1}^n  \tx_i ^{\alpha _i} \geq \prod_{k=1}^n  E_{k, l_k}^{-\alpha _{l_k}}  \| \bbx \|  \overset{\text{(c)}}{=} \| \bm{B}^{-1} \bbx \|,\notag
\end{align}
where (a) is true thanks to the fact that in the sum there is only one nonzero term, which comes from identity (i) and footnote \ref{footnote:k,ik,lk}; (b) holds because $\alpha_k = \alpha_{l_k}$ for all $k$; (c) comes from identity (ii).
		Hence, $\bm{A}^{-1}\bm{x}\in \gK $. This implies \( \bm{A}\gK \supseteq \gK \) and consequently \( \bm{A}\gK = \gK \).

		Now, it remains to show \eqref{eq:claims-E-inv-B-inv}.
		Since \( \bm{E} \) is a generalized permutation matrix with all nonzero elements being positive, then we immediately have (i) from \( \bm{E}\bm{E}^{-1} = \bm{I}_n \).
		Recall that, by assumption, \( \| \bm{B}\bm{x} \| = \prod_{k=1}^n E_{k, l_k}^{\alpha _{l_k}} \| \bm{x} \| \) for any \( \bm{x} \in \R^m \) and  \( \bm{B} \) is invertible.
		Using these, we can deduce (ii) in \eqref{eq:claims-E-inv-B-inv} as follows: for any \( \bm{x} \in \R^m \),\vspace{-0.1 cm}
		\[ \vspace{-0.1 cm}
\| \bm{x} \| = \| \bm{B}\bm{B}^{-1} \bm{x} \| = \prod_{k=1}^n  E_{k, l_k} ^{\alpha _{l_k}} \| \bm{B}^{-1} \bm{x} \|.
\]
	\end{proof}

	The next theorem is about the dimension of \( \aut(\gK) \).
	\begin{theorem}\label{thm:dimension-Aut(P)}
		Let \( m \geq 1 \), \( n \geq 2 \) and \( \bm{\alpha} \in (0, 1)^n \) such that \( \sum_{i=1}^n \alpha _i = 1 \), then we have the following statements about \( \dim\aut(\gK) \).
		\begin{enumerate}[label={(\roman*)}]
			\item\label{item:dimension-Aut(P)-n=2} If \( m \geq 1 \), \( n = 2 \) and \( \bm{\alpha} := (1 / 2, 1 / 2) \), then \( \dim\aut(\gK) = (m^2 + 3m + 4) / 2 \).
			\item\label{item:dimension-Aut(P)-nonsymmetric} If \( m \geq 1 \), \( n > 2 \) and \( \sum_{i=1}^n \alpha _i = 1 \) or \( m \geq 1 \), \( n = 2 \), \( \alpha_1 \neq \alpha _2 \) and \( \alpha_1 + \alpha _2 = 1 \), then:\vspace{-0.1 cm}
			\begin{equation}\vspace{-0.1 cm}
				\label{eq:form-lie-aut-power-cone}
				\lie\aut(\gK) = \left\{
				\begin{bmatrix}
					\bm{G} & \bm{0}  \\
					\bm{0}  & \diag(\bm{h})
				\end{bmatrix} \Bigg| \,
                \begin{split}
				  & \bm{G} + \bm{G}^{\top}  = 2\bm{\alpha}^{\top} \bm{h} \bm{I}_m, \\
				  & \bm{G} \in \R^{m\times m}, \, \bm{h}\in \R^n
                \end{split}
				\right\}.
			\end{equation}
			Hence, \( \dim\aut(\gK) = \dim\lie\aut(\gK) = n + m(m - 1) / 2 \).
		\end{enumerate}
	\end{theorem}
	\begin{proof}
		\ref{item:dimension-Aut(P)-n=2} If \( m \geq 1 \), \( n = 2 \) and \( \bm{\alpha} := (1 / 2, 1 / 2) \), then \( \gK \) is isomorphic to a second-order cone; see, \cite[Section~3.1.2]{MC2020}.
			Hence, we know from \cite[Page 12 (v)]{GT2014} that\vspace{-0.1 cm}
			\[ \vspace{-0.1 cm}
\dim\aut(\gK) = \frac{(m+2)^2 - m }{2} = \frac{m^2 + 3m + 4}{2}.
\]

			\ref{item:dimension-Aut(P)-nonsymmetric} By \cite[Corollary 3.45]{H2015}, \( \dim \aut(\gK) = \dim\lie\aut(\gK) \).
			This in addition to \cite[Corollary 3.46]{H2015} show that it suffices to calculate the dimension of the tangent space at the identity of \( \aut(\gK) \) to obtain \( \dim\aut(\gK) \).
			
			First, we compute $\lie \aut(\gK)$ and for that we consider an arbitrary continuously differentiable curve \( {\bm F} \!:\! (-1, 1) \!\to\! \aut(\gK) \) with \( {\bm F}(0) = \bm{I}_{m + n} \) and \( {\bm F}(t) \in \aut(\gK) \) for any \( t \in (-1, 1) \).
			We further denote
			\[ {\bm F}(t) =
			\begin{bmatrix}
				\bm{G}_t    & \bm{0} \\
				\bm{0} & \bm{H}_t
			\end{bmatrix} \quad \text{ and } \quad \dot{{\bm F}}(t) =
			\begin{bmatrix}
				\dot{\bm{G}}_t & \bm{0}    \\
				\bm{0}    & \dot{\bm{H}}_t
			\end{bmatrix},
			\]
			where \( \bm{G}_t \in \R^{m \times m} \) and \( \bm{H}_t \in \R^{n \times n} \) are both invertible; \( \bm{G}_0 = \bm{I}_m \), \( \bm{H}_0 = \bm{I}_n \); \( \bm{H}_t \) is a generalized permutation matrix with all nonzero elements being strictly positive (which we assume, by suitably shrinking the neighborhood of definition of ${\bm F}$ and reparameterizing, to be only nonzero along the diagonal); \( \dot{{\bm F}}(0) \) lies in the tangent space of \( \aut(\gK) \) at \( \bm{I} \), {that is,
			\begin{equation}
			  \label{eq:F(0)-in-Lie-Aut-gK}
			  \dot{\bm{F}}(0) =
			\begin{bmatrix}
				\dot{\bm{G}}_0 & \bm{0}    \\
				\bm{0}    & \dot{\bm{H}}_0
			\end{bmatrix} \in \lie\aut(\gK);
			\end{equation}}
		  \( \dot{\bm{G}}_t \) and \( \dot{\bm{H}}_t \) refer to the {componentwise} derivative of \( \bm{G} \) and \( \bm{H} \) with respect to \( t \), respectively.

{Since $\bm{H}_t$ and $\dot{\bm{H}}_t$ are diagonal, we let $\bm{h}_t$ and $ \dot{\bm{h}}_t $ be the diagonal vectors of $ \bm{H}_t $ and $ \dot{\bm{H}}_t $, respectively, i.e., $\bm{H}_t = \diag({\bm h}_t)$ and  $\dot{\bm{H}}_t = \diag(\dot{\bm h}_t)$.} We also let $h^k_t$ and $\dot h^k_t$ denote the $k$-th element of the vectors ${\bm h}_t$ and $\dot{\bm h}_t$ respectively.
			Then, from Theorem \ref{thm:self-duality-homogeneity},\vspace{-0.1 cm}
			\begin{equation}\label{eq:relationship-between-gt-ht}\vspace{-0.1 cm}
				\| \bm{G}_t\bm{x} \|^2 = \prod_{k=1}^n ( h_t^k )^{2\alpha_k} \| \bm{x} \|^2, \quad \forall \, \bm{x} \in \R^m, \, \forall \, t \in (-1, 1).
			\end{equation}
			Differentiating\footnote{This calculation simply uses the chain rule to differentiate $(h_t^k)^{2\alpha_{k}}$ for a given $k$, and then applies the product rule for the product over all $k$.} both sides of \eqref{eq:relationship-between-gt-ht} with respect to $t$, we can obtain
			\begin{align}
				& 2\bm{x}^{\top}  \bm{G}_t ^{\top} \dot{\bm{G}}_t \bm{x} = \bm{x}^{\top} \bm{x} \sum_{k=1}^n  2\alpha_{k}  \left( h_t^k \right) ^{2\alpha _{k} - 1} \dot{h}_t^k \prod_{j \neq k} \left( h_t^j \right)^{2\alpha _{j}} = \bm{x}^{\top} \bm{x} \sum_{k=1}^n  2\frac{\alpha _{k}}{h_t^k} \dot{h}_t^k \prod_{j=1}^n \left( h_t^j \right)^{2\alpha _{j}} \notag\\
				& = 2\bigg( \bm{x}^{\top} \bm{x} \prod_{j=1}^n ( h_t^j )^{2\alpha _{j}} \bigg) \sum_{k=1}^n  \frac{\alpha _{k}}{h_t^k} \dot{h}_t^k \stackrel{\rm (a)}{=} 2 \bm{x}^{\top} \bm{G}_t^{\top} \bm{G}_t \bm{x} \left( \bm{\alpha }\circ({\bm h}_t)^{-1} \right)^{\top} \dot{\bm h}_t ,\notag
			\end{align}
			{where the inverse is taken componentwise}, and the rest of (a) comes from \eqref{eq:relationship-between-gt-ht}.
			Notice that \( \left( \bm{\alpha }\circ({\bm h}_t)^{-1} \right)^{\top} \dot{\bm h}_t \) is a scalar, by rearranging terms, one has
			\[ \bm{x}^{\top} \left[  \bm{G}_t ^{\top} \dot{\bm{G}}_t - \left( \bm{\alpha }\circ({\bm h}_t)^{-1} \right)^{\top} \dot{\bm h}_t \bm{G}_t^{\top} \bm{G}_t   \right]\bm{x} = 0, \quad \forall \, \bm{x} \in \R^m, \ \forall \, t \in (-1, 1). \]
			Letting \( t = 0 \) and recalling \( \bm{G}_0 = \bm{I}_m, \bm{H}_0 = \bm{I}_n \), we have
			\begin{equation}\label{eqn:g0symmetrizing}
				\bm{x}^{\top} \left( \dot{\bm{G}}_0 -  \bm{\alpha}^{\top} \dot{\bm h}_0 I_m \right) \bm{x} = 0, \quad \forall \, \bm{x} \in \R^m.
			\end{equation}
Recall that $2\bm{x}^{\top} \dot{\bm{G}}_0 \bm{x} = \bm{x}^{\top} (\dot{\bm{G}}_0+\dot{\bm{G}}_0^{\top}) \bm{x}$. We can thus rewrite \eqref{eqn:g0symmetrizing} as
			\[ \bm{x}^{\top} \left( \dot{\bm{G}}_0+\dot{\bm{G}}_0^{\top} - 2\bm{\alpha}^{\top} \dot{\bm h}_0 I_m \right) \bm{x} = 0, \quad \forall \, \bm{x} \in \R^m. \]
			Since the matrix in the parentheses is zero, the above display implies that
			\[ \dot{\bm{G}}_0 + \dot{\bm{G}}_0^{\top} = 2 \bm{\alpha}^{\top} \dot{\bm h}_0 I_m. \]
{The above derivation and \eqref{eq:F(0)-in-Lie-Aut-gK} show that any matrix in $\lie\aut(\gK)$ satisfies the above display.}

Conversely, suppose that $\bm{G}$ and $\diag(\bm h)$ are such that $\bm{G}+\bm{G}^\top  = 2\bm{\alpha}^{\top} \bm{h} I_m$ and
$\bm{U} := 	\begin{bmatrix}
\bm{G} & \bm{0}  \\
\bm{0}  & \diag(\bm{h})
\end{bmatrix} $. We need to show that the matrix exponential $e^{t\bm{U}}$ belongs to $\aut(\gK)$  for every $t \in \R$.
To this end, recall that $e^{\bm{X}+\bm{Y}} = e^{\bm{X}}e^{\bm{Y}}$ {if $\bm{X}\bm{Y} = \bm{Y}\bm{X}$}, we have
\[
e^{t\bm{G}} = e^{2t\bm{\alpha}^{\top} \bm{h} I_m - t\bm{G}^\top} = e^{2t\bm{\alpha}^{\top} \bm{h} \bm{I}_m} e^{-t\bm{G}^\top} = e^{2t\bm{\alpha}^{\top} \bm{h}} e^{- t\bm{G}^\top},
\]
{since $2t\bm{\alpha}^{\top} \bm{h} \bm{I}_m$ and $-t\bm{G}^\top$ commute}.
This shows that $(e^{t\bm{G}})^{\top}e^{t\bm{G}} =\!e^{t\bm{G}^\top}e^{t\bm{G}}
= e^{2t\bm{\alpha}^{\top} \bm{h}} \bm{I}_m$, i.e., $e^{t\bm{G}}$ is an orthogonal matrix multiplied by the scalar $e^{t\bm{\alpha}^{\top} \bm{h}}$.
Then\vspace{-0.15 cm} 
\begin{equation}\label{hahahahaha}\vspace{-0.15 cm}
\|e^{t\bm{G}}\bm{x}  \| = e^{t \bm{\alpha}^{\top}\bm{h}}\| \bm{x} \| =
e^{\sum_{i=1}^n t {h}_i {\alpha}_i}\| \bm{x} \|=  \prod _{i=1}^n {(e^{t{h}_i})}^{\alpha_i}  \| \bm{x} \|\ \ \ \forall \bm x \in \R^m.
\end{equation}
Since\vspace{-0.1 cm}
\begin{equation*}\vspace{-0.1 cm}
e^{t\bm{U}} = \begin{bmatrix}
e^{t\bm{G}} & \bm{0}         \\
\bm{0}  & e^{\diag(t\bm{h})}
\end{bmatrix} = \begin{bmatrix}
e^{t\bm{G}} & \bm{0}         \\
\bm{0}  & {{\diag(e^{t\bm{h}})}}
\end{bmatrix},
\end{equation*}
where $e^{t\bm{h}}$ corresponds to the vector such that its $i$-th component is $e^{th_i}$ and $h_i$ is the $i$-th component of $\bm{h}$, we conclude from \eqref{hahahahaha} and Theorem~\ref{thm:self-duality-homogeneity} that
$e^{t\bm{U}} \in \aut(\gK)$.

Finally, a direct computation shows that the dimension of the right-hand side of \eqref{eq:form-lie-aut-power-cone} is $n+m(m-1)/2$, which is just the claimed dimension.
	\end{proof}

	\subsection{Homogeneity, irreducibility and perfectness of generalized power cone}\label{subsec:nonhomogeneity-irreducibility-perfectness}
	In this subsection, we will use Theorem \ref{thm:self-duality-homogeneity} to prove the homogeneity, irreducibility and perfectness of $\gK$.
	Before moving on, we recall the following lemma.
	\begin{lemma}\label{lemma:reducibility-perfectness}
		\begin{enumerate}[label=(\roman*)]
			\item\label{item:reducibility-dim-face} If a closed convex pointed cone \( \stdCone \) is reducible, i.e., \( \stdCone \) is a direct sum of two nonempty, nontrivial sets \( \stdCone_1 \) and \( \stdCone_2 \), then we have \( \stdCone_1 \properideal \stdCone, \, \stdCone_2 \properideal \stdCone \) and \( \dim(\stdCone) = \dim(\stdCone_1) + \dim(\stdCone_2) \).
			\item\label{item:perfectness-equivalence} A proper cone \( \stdCone \subseteq \R^p \) is perfect if and only if \( \dim\lie\aut(\stdCone) \geq p \).
		\end{enumerate}
	\end{lemma}

	\begin{proof}
		\ref{item:reducibility-dim-face}
		The fact that $\stdCone_1$ and $\stdCone_2$ are faces is well-known, see \cite[Lemma~3.2]{LS75}. The conclusion on dimensions follows directly from the definition of direct sum.


			\ref{item:perfectness-equivalence} This fact comes from~\cite[Theorem 1]{OG2016} and the first display on~\cite[Page 4]{GT2014}.
	\end{proof}

	Using Lemma~\ref{lemma:reducibility-perfectness}, Theorems~\ref{thm:self-duality-homogeneity} and~\ref{thm:dimension-Aut(P)}, we have the following corollary.
	\begin{corollary}\label{corollary:irreducibility-nonhomogeneous-perfect}
		Let \( m \geq 1 \), \( n \geq 2 \) and \( \bm{\alpha} \in (0, 1)^n \) such that \( \sum_{i=1}^n \alpha _i = 1 \), then the following statements hold for the generalized power cone \( \gK \).
		\begin{enumerate}[label=(\roman*)]
			\item\label{item:self-duality-irreducibility} \( \gK \) is irreducible.
			\item\label{item:homogeneity-perfectness-symmetric} If \( m \geq 1, n = 2 \) and \( \bm{\alpha} := (1 / 2, 1 / 2) \), then \( \gK \) is homogeneous and perfect.
			\item\label{item:nonhomogeneity-(im)perfectness-nonsymmetric} If \( m \geq 1 \), \( n > 2 \) and \( \sum_{i=1}^n \alpha _i = 1 \) or \( m \geq 1 \), \( n = 2 \), \( \alpha_1 \neq \alpha _2 \) and \( \alpha_1 + \alpha _2 = 1 \), then \( \gK \) is nonhomogeneous.
			In addition, if \( 1 \leq m \leq 2 \), then \( \gK \) is not perfect; if \( m \geq 3 \), then \( \gK \) is perfect.
		\end{enumerate}
	\end{corollary}

	\begin{proof}
		\ref{item:self-duality-irreducibility}
			Recall that the two types of faces of \( \gK \) are defined as in~\eqref{eq:1st-type-face-ray} and~\eqref{eq:2nd-type-face-d-dim}, with dimensions being \( 1 \) and \( n - | \stdIndex | \), respectively.
			Since \( \stdIndex \neq \emptyset\) and so \( | \stdIndex | \geq 1 \), for any possible pair of nontrivial faces \( \stdFace_1 \) and \( \stdFace_2 \) of \( \gK \), we have
			$
\dim(\stdFace_1) + \dim(\stdFace_2) < m + n = \dim(\gK)
$.
			This together with Lemma~\ref{lemma:reducibility-perfectness}\ref{item:reducibility-dim-face} show that \( \gK \) is irreducible.

			\ref{item:homogeneity-perfectness-symmetric}
			If \( m \geq 1, n = 2 \) and \( \bm{\alpha} := (1 / 2, 1 / 2) \), \( \gK \) is isomorphic to a second-order cone and so is homogeneous; see, for example, \cite[Section~3.1.2]{MC2020}.
			The perfectness holds by Theorem~\ref{thm:dimension-Aut(P)}\ref{item:dimension-Aut(P)-n=2} and Lemma~\ref{lemma:reducibility-perfectness}\ref{item:perfectness-equivalence}.

			\ref{item:nonhomogeneity-(im)perfectness-nonsymmetric}
			Take any \( m \geq 1, n > 2 \) with any \( \bm{\alpha} \in (0, 1)^n \) such that \( \sum_{i=1}^n \alpha _i = 1  \) or \( m \geq 1, n = 2 \) with any \( \bm{\alpha} \in (0, 1)^2 \) such that \( \alpha_1 \neq \alpha _2 \), consider \( \bm{x} = (\bm{0}, \tbx) \in \ri \gK \) and \( \bm{y} = (\bby, \tby) \in \ri \gK \), where \( \min_i\{\tx_i\} > 0\), \(\min_i\{\ty\} > 0 \) and \( \bby \neq \bm{0}, \,\| \bby \|  < \prod_{i=1}^n \ty_i ^{\alpha _i}\). Using \eqref{eq:form-A-AK=K}, for all \( \bm{A} \) such that \( \bm{A} \gK = \gK \), we have \( \bm{A}\bm{x} \neq \bm{y} \) because \( \bm{B} \bm{0} = \bm{0} \neq \bby \) for all possible \( \bm{B} \).
			Then by definition, \( \gK \) with \( m \geq 1 \), \( n = 2 \) and \( \sum_{i=1}^n \alpha _i = 1 \) or \( m \geq 1 \), \( n = 2 \), \( \alpha_1 \neq \alpha _2 \) and \( \alpha_1 + \alpha _2 = 1 \) is nonhomogeneous.
			By Theorem~\ref{thm:dimension-Aut(P)}\ref{item:dimension-Aut(P)-nonsymmetric}, we have \( \dim\lie\aut(\gK) = n + \frac{m(m-1)}{2} \geq m + n \) if and only if \( m \geq 3 \). The conclusion concerning perfectness now follows from
			this and Lemma~\ref{lemma:reducibility-perfectness}\ref{item:perfectness-equivalence}.
	\end{proof}
\section*{Acknowledgements}
{
  We thank the referees for their  comments, which helped to improve the paper.
}


	\bibliographystyle{abbrvurl}
	\bibliography{bib_plain}
\end{document}